\documentclass{amsart}

\usepackage{amsmath}
\usepackage{amssymb}
\usepackage{bbm}
\usepackage{enumerate}
\usepackage{appendix}
\usepackage{mathrsfs}
\usepackage{hyperref}

\usepackage[T1]{fontenc}
\usepackage[utf8]{inputenc}

\numberwithin{equation}{section}

\newcommand{\B}{\boldsymbol{B}}

\newcommand{\R}{\mathbb{R}}
\newcommand{\N}{\mathbb{N}}
\newcommand{\Z}{\mathbb{Z}}

\newcommand{\E}{\mathbb{E}}
\newcommand{\pr}{\mathbb{P}}

\newcommand{\e}{\operatorname{e}}
\newcommand{\dd}{\,{\mathrm d}}
\newcommand{\db}{{\mathrm d}}

\newcommand{\im}{\operatorname{i}}
\newcommand{\ind}{\mathbbm{1}}

\newcommand{\eps}{\varepsilon}

\newtheorem{lemma}{Lemma}[section]

\newtheorem{propn}[lemma]{Proposition}
\newtheorem{thm}[lemma]{Theorem}

\newtheorem{defn}[lemma]{Definition}
\newtheorem{remark0}[lemma]{Remark}
\newtheorem{eg0}[lemma]{Example}

\newenvironment{remark}{\begin{remark0}\rm}{\hspace*{\fill} $\square$
                        \end{remark0}}

\author[K. Habermann]{Karen Habermann}
\address{University of Bonn, Hausdorff Center for Mathematics,
  Endenicher Allee 62, 53115 Bonn, Germany.}
\email{habermann@iam.uni-bonn.de}
\thanks{Research supported by the German Research Foundation DFG
  through the Hausdorff Center for Mathematics.}

\title{A semicircle law and decorrelation phenomena for iterated
  Kolmogorov loops}
\begin{document}
\begin{abstract}
  We consider a standard one-dimensional Brownian motion on
  the time interval $[0,1]$
  conditioned to have vanishing iterated time integrals up to order
  $N$. We show that the resulting processes can be expressed explicitly
  in terms of shifted Legendre polynomials and the original Brownian
  motion, and we use these representations to prove that the processes
  converge weakly as $N\to\infty$ to the zero process.
  This gives rise to a polynomial decomposition for Brownian motion.
  We further
  study the fluctuation processes obtained through scaling by
  $\sqrt{N}$ and show that they converge in finite dimensional
  distributions as $N\to\infty$ to a collection of independent
  zero-mean Gaussian random variables whose variances follow a scaled
  semicircle. The fluctuation result is a consequence of a limit
  theorem for Legendre polynomials which quantifies their completeness
  and orthogonality property.
  In the proof of the latter, we encounter a Catalan triangle.
\end{abstract}

\maketitle
\thispagestyle{empty}

\section{Introduction}
Let $(B_t)_{t\in[0,1]}$ be a Brownian motion in $\R$, which we assume
is realised as the coordinate process on the path space $\{w\in
C([0,1],\R)\colon w_0=0\}$ under Wiener measure $\pr$. The
stochastic process in $\R^2$ which pairs the standard
one-dimensional Brownian motion $(B_t)_{t\in[0,1]}$
with its time integral is the Kolmogorov diffusion, named after
Kolmogorov~\cite{kolmogorov}. Similarly, pairing Brownian motion
with its iterated time integrals up to some order
gives rise to the iterated Kolmogorov diffusion.
\begin{defn}
  Let $N\in\N$. The stochastic process $(\B_t^N)_{t\in[0,1]}$ in
  $\R^N$ defined by
  \begin{equation*}
    \B_t^N=\left(
      B_t,\int_0^tB_{s_1}\dd s_1,
      \int_0^t\int_0^{s_2}B_{s_1}\dd s_1\dd s_2,\dots,
      \int_0^t\int_0^{s_{N-1}}\dots \int_0^{s_2} B_{s_1}\dd s_1\dots\dd s_{N-1}
    \right)
  \end{equation*}
  is the iterated Kolmogorov diffusion of step $N$.
\end{defn}
In particular, $(\B_t^1)_{t\in[0,1]}$ is simply the Brownian motion
$(B_t)_{t\in[0,1]}$ and $(\B_t^2)_{t\in[0,1]}$ is the associated
Kolmogorov diffusion. Since $(\B_t^N)_{t\in[0,1]}$ is a Gaussian
process for all $N\in\N$, we can make sense of conditioning the
process $(\B_t^N)_{t\in[0,1]}$ on $\B_1^N=0$. Considering the first
component of the resulting process shows the existence of the iterated
Kolmogorov loop of step $N$.
\begin{defn}
  The iterated Kolmogorov loop of step $N\in\N$ is the stochastic process
  in $\R$ obtained by conditioning $(B_t)_{t\in[0,1]}$ on $\B_1^N=0$.
\end{defn}
The terminology is motivated by Baudoin~\cite[Section~3.6]{baudoin_flows}
where Brownian motion in $\R^d$ conditioned to have trivial
truncated signature of order $N$ is called the Brownian loop of step $N$.

We study the iterated Kolmogorov loops of step $N$ in the limit
$N\to\infty$. Our analysis exploits the explicit expression below
for iterated Kolmogorov loops
in terms of shifted Legendre polynomials. For the proof and
further discussions, see Section~\ref{kol_loops}.
\begin{propn}\label{propn:kol_loop_leg}
  Let $Q_n$ be the shifted Legendre polynomial
  of degree $n$ on $[0,1]$. For $N\in\N$,
  the stochastic process $(L_t^N)_{t\in[0,1]}$ in $\R$ defined by
  \begin{equation}\label{expr4L}
    L_t^N=
    B_t-\sum_{n=0}^{N-1}(2n+1)\int_0^tQ_n(r)\dd r \int_0^1Q_n(r)\dd B_r
  \end{equation}
  has the same law as the iterated Kolmogorov loop of step $N$.
\end{propn}
As a consequence of the completeness and orthogonality of the shifted
Legendre polynomials, we obtain a law of large numbers type
theorem for the iterated Kolmogorov loops. This result can be
rephrased to give a polynomial decomposition of Brownian motion,
cf. Section~\ref{kol_loops}.
\begin{thm}\label{thm:LLN}
  Let $\Omega^{0,0}=\{w\in C([0,1],\R)\colon w_0=w_1=0\}$ be
  the set of continuous loops in $\R$ at zero. The laws of the iterated
  Kolmogorov loops of step $N$ converge weakly on $\Omega^{0,0}$ as
  $N\to\infty$ to the unit mass $\delta_0$ at the zero path.
\end{thm}
A similar question can be posed for Brownian loops.
We conjecture that, for $d\geq 2$,
the Brownian loops of step $N$ converge weakly to the zero
process in $\R^d$ as $N\to\infty$, see~\cite[Conjecture~4.1.3]{PhDthesis}.

Going beyond the law of large numbers, we further study the
fluctuation processes of the iterated Kolmogorov loops of step $N$
obtained through scaling by $\sqrt{N}$ in the limit $N\to\infty$.
As seen in Section~\ref{kol_loops}, cf. Lemma~\ref{lem:covariance}, 
the covariance function $C_N$ of the iterated Kolmogorov loop
of step $N$ is given, for $s,t\in[0,1]$, by
\begin{equation*}
  C_N(s,t)=
  \min(s,t)-\sum_{n=0}^{N-1}(2n+1)\int_0^sQ_n(r)\dd r\int_0^tQ_n(r)\dd r\;.
\end{equation*}
Our central limit type theorem for the iterated Kolmogorov loops then
relies on the following limit theorem involving Legendre
polynomials. For convenience, it is expressed in terms of the Legendre
polynomials on $[-1,1]$.
\begin{thm}\label{thm:CLT4P}
  Let $P_n$ be the Legendre polynomial of degree $n$
  on $[-1,1]$. Fix $x,y\in[-1,1]$ and, for $N\in\N$, set
  \begin{equation}\label{defn:R}
    R_N(x,y)=N\left(\min(1+x,1+y)-\sum_{n=0}^{N-1}\frac{2n+1}{2}
      \int_{-1}^xP_n(z)\dd z
      \int_{-1}^yP_n(z)\dd z\right)\;.
  \end{equation}
  Then, we have
  \begin{equation*}
    \lim_{N\to\infty} R_N(x,y)=
    \begin{cases}
      \frac{1}{\pi}\sqrt{1-x^2} & \mbox{if } x=y\\
      0 & \mbox{if } x\not=y
    \end{cases}\;,
  \end{equation*}
  that is, $R_N\colon[-1,1]\times[-1,1]\to\R$ converges pointwise
  as $N\to\infty$ to the specified limit function.
\end{thm}
This result quantifies an integrated version of the completeness
and orthogonality property for the Legendre polynomials, which
in terms of the Dirac delta function is stated,
for $x,y\in[-1,1]$, as
\begin{equation*}
  \sum_{n=0}^\infty\frac{2n+1}{2} P_n(x)P_n(y)
  =\delta(x-y)\;.
\end{equation*}
The proof of Theorem~\ref{thm:CLT4P} is split into an on-diagonal and
an off-diagonal analysis. The pointwise convergence on the diagonal
follows from a convergence of moments,
cf. Proposition~\ref{propn:moments}, and a locally uniform
convergence implied by Lemma~\ref{lem:locunif_bound} and the
Arzel{\`a}--Ascoli theorem, whereas the pointwise convergence
away from the diagonal relies on a Christoffel--Darboux type formula
for the integrals of the Legendre polynomials,
cf. Proposition~\ref{propn:chris4I}. In both parts, we use asymptotic
estimates for Legendre polynomials and their integrals which are
implied by the Darboux formula for Jacobi polynomials. For
convenience, we include the Darboux formula as Theorem~\ref{thm:darboux}.

Equipped with Theorem~\ref{thm:CLT4P}, we deduce a central limit type
theorem for the iterated Kolmogorov loops. With
Proposition~\ref{propn:kol_loop_leg} in mind, we consider the
processes $(F_t^N)_{t\in[0,1]}$ defined by $F_t^N=\sqrt{N} L_t^N$.
\begin{thm}\label{thm:CLT}
  The fluctuation processes $(F_t^N)_{t\in[0,1]}$ converge in
  finite dimensional distributions as $N\to\infty$
  to the collection $(F_t)_{t\in[0,1]}$
  of independent zero-mean Gaussian random variables whose
  variances are given, for $t\in[0,1]$, by
  \begin{equation*}
    \E\left[F_t^2\right]=\frac{1}{\pi}\sqrt{t(1-t)}\;.
  \end{equation*}
\end{thm}
It is certainly interesting that the variances of the limit fluctuations
follow a scaled semicircle, and we remark that semicircles naturally
appear in other limit theorems such as the Wigner semicircle law
in random matrix theory,
cf.~\cite[Theorem~2.1.1]{guionnet}, or the central limit theorem in free
probability, see~\cite[Theorem~8.10]{speicher}.
Moreover, as pointed out in Remark~\ref{rem:decorr}, we can obtain a
non-trivial bound on the scale of the decorrelation observed for the
rescaled iterated Kolmogorov loops.

The reason for considering convergence in finite dimensional
distributions in Theorem~\ref{thm:CLT} is that
while the collection $(F_t)_{t\in[0,1]}$ of independent zero-mean
Gaussian random variables is well-defined,
see~\cite[Section~2.3]{bogachev}, it
neither has a realisation as a process in $C([0,1],\R)$,
cf.~\cite[Example~1.2.4]{kallianpur},
nor is it equivalent to a measurable process,
cf.~\cite[Example~1.2.5]{kallianpur}.
This is also why $(F_t)_{t\in[0,1]}$, which could be thought of as an
inhomogeneous white noise process with vanishing power spectral
density, is not treated as a useful mathematical model for white noise.

The paper is organised as follows. In Section~\ref{legendre}, we
recall properties of Legendre polynomials and their integrals, and
we introduce complex-valued polynomials which simplify the
presentation and some
of the arguments given in Section~\ref{moments}. That section is concerned
with studying the moments of $R_N$ on the diagonal in the 
limit $N\to\infty$. As part of the analysis,
which uses partial fraction decompositions,
we encounter a Catalan triangle,
see Remark~\ref{rem:catalan}. In Section~\ref{kol_loops}, we determine an
expression for the iterated Kolmogorov loop of step $N$ in terms of
the inverse of an $N\times N$ factorial Hankel matrix, and we prove
Proposition~\ref{propn:kol_loop_leg} as well as
Theorem~\ref{thm:LLN}.
In Section~\ref{fluctuations}, we give the proof of
Theorem~\ref{thm:CLT4P} which makes use of the Christoffel--Darboux
type formula for the integrals of the Legendre polynomials
stated in Proposition~\ref{propn:chris4I}, and we conclude with the
proof of Theorem~\ref{thm:CLT}.  Throughout, we use
the convention that $\N$ denotes the positive
integers, whereas $\N_0$ refers to the non-negative integers.

\proof[Acknowledgement]
I am grateful to Martin Huesmann and James Norris for helpful
discussions.
\section{Legendre polynomials and their integrals}
\label{legendre}
We discuss properties of Legendre polynomials that are needed for our
subsequent analysis and we extend the Legendre
polynomials to a family of complex-valued polynomials on $[-1,1]$.
Using this extension, we introduce a second family of complex-valued
polynomials, which is linked to the integrals of Legendre polynomials.

Let $\{P_n\colon n\in\N_0\}$ be the family of the Legendre polynomials on
the interval $[-1,1]$. Following the physical motivation presented in
Arfken and Weber~\cite[Section~12.1]{arfken} of considering the
electrostatic potential of a point charge, the Legendre polynomials
can be defined by means of a generating
function through
\begin{equation*}
  \sum_{n=0}^\infty P_n(x)z^n =\frac{1}{\sqrt{1-2xz+z^2}}
  \quad\mbox{for }z\in(-1,1)\;.
\end{equation*}
As derived in~\cite[Section~12.2]{arfken}, the generating
function can be used to establish the Bonnet recursion formula
\begin{equation}\label{bonnet}
  (n+1)P_{n+1}(x)=(2n+1)xP_n(x)-nP_{n-1}(x)
  \quad\mbox{for }n\in\N \mbox{ and }x\in[-1,1]
\end{equation}
as well as the relation
\begin{equation}\label{LegInt}
  (2n+1)\int_{-1}^xP_n(z)\dd z=P_{n+1}(x)-P_{n-1}(x)
  \quad\mbox{for }n\in\N \mbox{ and }x\in[-1,1]\;.
\end{equation}
It is further shown in~\cite[Section~12.2]{arfken} that we have the
parity property 
\begin{equation}\label{eq:parity}
  P_n(-x)=(-1)^nP_n(x)
  \quad\mbox{for }n\in\N_0 \mbox{ and }x\in[-1,1]\;,
\end{equation}
and that, for all $n\in\N_0$, the Legendre polynomial $P_n$ satisfies
the Legendre differential equation
\begin{equation}\label{LegODE}
  \frac{\db}{\db x}\left(
    \left(1-x^2\right)\frac{d}{\db x}\right) P_n(x)
  +n(n+1)P_n(x)=0\;.
\end{equation}
The latter could also be used to define the Legendre polynomials by
letting $P_n$ be the polynomial solution of the Legendre
differential equation~(\ref{LegODE}).
As detailed in Lebedev~\cite[Section~4.5]{lebedev},
the orthogonality of the Legendre polynomials
\begin{equation*}
  \int_{-1}^1P_n(x)P_m(x)\dd x=0
  \quad\mbox{for }n,m\in\N_0\mbox{ with }n\not=m
\end{equation*}
follows from (\ref{LegODE}) and is applied together with the Bonnet
recursion formula~(\ref{bonnet}) to prove that
\begin{equation}\label{orth4P}
  \int_{-1}^1\left(P_n(x)\right)^2\dd x=\frac{2}{2n+1}
  \quad\mbox{for }n\in\N_0\;.
\end{equation}
Alternatively, Legendre polynomials could be defined as the sequence
of polynomials orthogonal with respect to the weighting function $1$
over $[-1,1]$ subject to requiring $P_n(1)=1$
for all $n\in\N_0$, see Andrews, Askey and Roy~\cite[Remark~5.3.1]{roy}.
The Legendre polynomials then arise by applying the Gram--Schmidt
orthogonalisation process to the monomials
$\{x^n\colon n\in\N_0\}$ on $[-1,1]$
with respect to the usual $L^2$ inner product and by
imposing the normalisation $P_n(1)=1$ for all $n\in\N_0$.
With this approach the completeness of the Legendre polynomials
follows immediately.
Another option is to rewrite the Legendre
differential equation~(\ref{LegODE})
as an eigenvalue problem and to
appeal to Sturm--Liouville theory, cf.~\cite[Chapter~10]{arfken}.

In our expressions for the iterated Kolmogorov loops, we actually need the
family $\{Q_n\colon n\in\N_0\}$ of the shifted Legendre polynomials on the
interval $[0,1]$, which are given by
\begin{equation*}
  Q_n(t)=P_n(2t-1)\quad\mbox{for }t\in[0,1]\;.
\end{equation*}
These polynomials inherit their properties from the
Legendre polynomials on $[-1,1]$.
In particular, the
shifted Legendre polynomials form a complete orthogonal system with
\begin{equation}\label{orthconst4Q}
  \int_0^1\left(Q_n(t)\right)^2\dd t =\frac{1}{2n+1}
  \quad\mbox{for }n\in\N_0\;,
\end{equation}
and they satisfy the parity relation
\begin{equation}\label{parityQ}
  Q_n(1-t)=(-1)^nQ_n(t)
  \quad\mbox{for }n\in\N_0 \mbox{ and }t\in[0,1]\;.
\end{equation}

\subsection{Complex-valued Legendre polynomials}
\label{complexext}
We introduce a family $\{P_n\colon n\in\Z\}$ indexed by
the integers $\Z$
of complex-valued polynomials on $[-1,1]$ which extends the 
family $\{P_n\colon n\in\N_0\}$ of Legendre polynomials on $[-1,1]$.
When generalising the Legendre polynomials and dealing with associated
Legendre polynomials, it is common to define the associated Legendre
polynomial of zeroth order and negative degree $-n-1$ to equal $P_n$
for $n\in\N_0$. The reason for this is that the Legendre differential
equation~(\ref{LegODE}) is invariant under a change from $n$ to
$-n-1$. However, we instead choose to set
\begin{equation}\label{defn:complexL}
  P_{-n-1}(x)=\im P_n(x) \quad\mbox{for }n\in\N_0
  \mbox{ and }x\in[-1,1]\;.
\end{equation}
Our motivation for this choice is that,
according to~(\ref{orth4P}), it gives rise to
\begin{equation*}
  \int_{-1}^1\left(P_{-n-1}(x)\right)^2\dd x
  =-\int_{-1}^1\left(P_{n}(x)\right)^2\dd x
  =-\frac{2}{2n+1}=\frac{2}{2(-n-1)+1}
  \quad\mbox{for }n\in\N_0\;,
\end{equation*}
and therefore, we have
\begin{equation}\label{scale4Z}
  \int_{-1}^1\left(P_n(x)\right)^2\dd x =\frac{2}{2n+1}
  \quad\mbox{for all }n\in\Z\;.
\end{equation}
Moreover, the Bonnet recursion formula extends consistently across
the original boundary at $n=0$ to all $n\in\Z$.
\begin{lemma}\label{lem:bonnet}
  For all $n\in\Z$ and all $x\in[-1,1]$, we have
  \begin{equation*}
    (n+1)P_{n+1}(x)=(2n+1)xP_n(x)-nP_{n-1}(x)\;.
  \end{equation*}
\end{lemma}
\begin{proof}
  For $n\in\N$, this is the usual Bonnet recursion
  formula~(\ref{bonnet}).
  If $n\in\Z\setminus\N_0$ then, due to~(\ref{defn:complexL}),
  we have
  \begin{equation*}
    P_n(x)=\im P_{-n-1}(x) \quad\mbox{for }x\in[-1,1]\;,
  \end{equation*}
  and we use~(\ref{bonnet}) in the form
  \begin{equation*}
    -nP_{-n}(x)=(-2n-1)xP_{-n-1}(x)-(-n-1)P_{-n-2}(x)
  \end{equation*}
  to deduce that, for all $x\in[-1,1]$,
  \begin{align*}
    (n+1)P_{n+1}(x)=(n+1)\im P_{-n-2}(x)
    &=(2n+1)x\im P_{-n-1}(x)-n\im P_{-n}(x)\\
    &=(2n+1)xP_n(x)-n P_{n-1}(x)\;,
  \end{align*}
  as required. For $n=0$, we explicitly  see that $P_1(x)=x$ coincides
  with $xP_0(x)=x$.
\end{proof}
This extension of the Legendre polynomials turns out to be convenient
for our analysis. In the next section, we use these polynomials to
introduce a family of complex-valued polynomials related to the
integrals of the Legendre polynomials.

\subsection{Integrals of Legendre polynomials}
\label{general_int}
Let $\{I_n\colon n\in\Z\}$ be the family index by $\Z$ defined by
\begin{equation}\label{defn:I}
  (2n+1)I_n(x)=P_{n+1}(x)-P_{n-1}(x)
  \quad\mbox{for }n\in\Z\mbox{ and }x\in[-1,1]\;.
\end{equation}
The property~(\ref{LegInt}) implies that
\begin{equation}\label{Iisint}
  I_n(x)=\int_{-1}^xP_n(z)\dd z
  \quad\mbox{for all }n\in\N \mbox{ and }x\in[-1,1]\;.
\end{equation}
However, we notice that this relation does not hold for $n=0$
because
\begin{equation*}
  I_0(x)=P_1(x)-P_{-1}(x)=P_1(x)-\im P_0(x)=x-\im\;,
\end{equation*}
whereas $\int_{-1}^xP_0(z)\dd z=1+x$. This discrepancy is exploited to
present a short proof of Lemma~\ref{lem:idat0}.

The parity property~(\ref{eq:parity}) yields
\begin{equation}\label{parityI}
  I_n(-x)=(-1)^{n+1}I_n(x)
  \quad\mbox{for }n\in\N \mbox{ and }x\in[-1,1]\;,
\end{equation}
and in particular,
\begin{equation}\label{Iboundary}
  I_n(1)=I_n(-1)=0
  \quad\mbox{for all }n\in\N\;.
\end{equation}
We further obtain the symmetry relation stated below as well as a
recursion formula.
\begin{lemma}\label{Isym}
  For all $n\in\N$ and all $x\in[-1,1]$, we have
  \begin{equation*}
    I_{-n-1}(x)=\im I_n(x)\;.
  \end{equation*}
\end{lemma}
\begin{proof}
  If $n\in\N$ then $n-1\in\N_0$ and therefore,
  by the definition~(\ref{defn:I}) and
  by~(\ref{defn:complexL}), we see that
  \begin{align*}
    -(2n+1)I_{-n-1}(x)
    =\left(2(-n-1)+1\right)I_{-n-1}(x)
    &=P_{-n}(x)-P_{-n-2}(x)\\
    &=\im P_{n-1}(x)-\im P_{n+1}(x)
    =-(2n+1)\im I_n(x)\;,
  \end{align*}
  which implies the desired result.
\end{proof}
In Lemma~\ref{Isym},
it is important to restrict our attention to $n\in\N$
since for $n=0$, we have
\begin{equation*}
  I_{-1}(x)=\im x-1
  \quad\mbox{and}\quad
  \im I_0(x)=\im x+1\;.
\end{equation*}
\begin{lemma}\label{lem:rec4I}
  For all $n\in\Z$ and all $x\in[-1,1]$, we have the recursion formula
  \begin{equation*}
    (n+2)I_{n+1}(x)=(2n+1)xI_n(x)-(n-1)I_{n-1}(x)\;.
  \end{equation*}
\end{lemma}
\begin{proof}
  This is a consequence of~(\ref{defn:I}) and the extended Bonnet
  recursion formula, cf. Lemma~\ref{lem:bonnet}. From
  \begin{equation*}
    (n+2)P_{n+2}(x)=(2n+3)xP_{n+1}(x)-(n+1)P_n(x)\;,
  \end{equation*}
  we deduce
  \begin{equation}\label{recstep1}
    (n+2)I_{n+1}(x)
    =\frac{(n+2)\left(P_{n+2}(x)-P_n(x)\right)}{2n+3}
    =xP_{n+1}(x)-P_n(x)\;,
  \end{equation}
  and similarly,
  \begin{equation*}
    n P_n(x)=(2n-1)xP_{n-1}(x)-(n-1)P_{n-2}(x)
  \end{equation*}
  implies that
  \begin{equation}\label{recstep2}
    (n-1)I_{n-1}(x)
    =\frac{(n-1)\left(P_n(x)-P_{n-2}(x)\right)}{2n-1}
    =P_n(x)-xP_{n-1}(x)\;.
  \end{equation}
  Adding equation~(\ref{recstep1}) to
  equation~(\ref{recstep2}) yields
  \begin{equation*}
    (n+2)I_{n+1}(x)+(n-1)I_{n-1}(x)=    
    x\left(P_{n+1}(x)-P_{n-1}(x)\right)\;,
  \end{equation*}
  and therefore, by~(\ref{defn:I}),
  \begin{equation*}
    (n+2)I_{n+1}(x)+(n-1)I_{n-1}(x)=(2n+1)xI_n(x)\;,
  \end{equation*}
  as claimed.
\end{proof}
Throughout the moment analysis presented in Section~\ref{moments}, it
is crucial, e.g. see Lemma~\ref{lem:mom_rec}, that the above
recursion formula holds for
all $n\in\Z$ and that the original boundary case at $n=0$ does not
need a special treatment.
For the latter,
the discrepancy between $I_0$ and the integral of
$P_0$ is also essential.

\subsection{Asymptotic behaviour}
We characterise the asymptotics in the limit $n\to\infty$
for Legendre polynomials and their
integrals by relating these polynomials to
Jacobi polynomials on $[-1,1]$ and then quoting the Darboux formula
for Jacobi polynomials.

Following Szeg\H{o}~\cite[Section~4.22]{gabor} and using the rising
Pochhammer
symbol, we define the Jacobi polynomial $P_n^{(\alpha,\beta)}$
of degree $n\in\N_0$ on $[-1,1]$ for $\alpha,\beta\in\R$ by
\begin{equation}\label{defn:jacobi}
  P_n^{(\alpha,\beta)}(x)
  =\frac{1}{n!}\sum_{k=0}^n\binom{n}{k}
  \left(n+\alpha+\beta+1\right)_k\left(\alpha+k+1\right)_{n-k}
  \left(\frac{x-1}{2}\right)^k
  \quad\mbox{for }x\in[-1,1]\;.
\end{equation}
If $\alpha,\beta>-1$, this agrees with the usual definition,
cf.~\cite[Definition~2.5.1]{roy},
\begin{equation*}
  P_n^{(\alpha,\beta)}(x)=\binom{n+\alpha}{n}
  {}_2F_1\left(-n,n+\alpha+\beta+1;\alpha+1;\frac{1-x}{2}\right)
  \quad\mbox{for }x\in[-1,1]\;,
\end{equation*}
where ${}_2F_1$ is the Gaussian hypergeometric function
represented by the power series
\begin{equation*}
  {}_2F_1(a,b;c;z)=\sum_{k=0}^\infty\frac{(a)_k(b)_k}{(c)_k}\frac{z^k}{k!}
  \quad\mbox{for }z\in(-1,1)\;.
\end{equation*}
For $\alpha,\beta>-1$ fixed, the polynomials $P_n^{(\alpha,\beta)}$
are orthogonal on
$[-1,1]$ with respect to the weighting function
$(1-x)^\alpha(1+x)^\beta$.
As discussed in~\cite[Section~4.21]{gabor}, the
expression~(\ref{defn:jacobi}) implies that
\begin{equation}\label{derjacobi}
  \frac{\db}{\db x}P_n^{(\alpha,\beta)}(x)
  =\frac{1}{2}(n+\alpha+\beta+1)P_{n-1}^{(\alpha+1,\beta+1)}(x)
  \quad\mbox{for }n\in\N \mbox{ and }x\in[-1,1]\;.
\end{equation}
As remarked in~\cite[Section~4.1]{gabor}, we further have
\begin{equation}\label{L2jacobi}
  P_n^{(0,0)}(x)=P_n(x)
  \quad\mbox{for }n\in\N_0 \mbox{ and }x\in[-1,1]\;,
\end{equation}
and from~(\ref{Iisint}) as well as (\ref{derjacobi}), it follows that
\begin{equation}\label{I2jacobi}
  P_{n+1}^{(-1,-1)}(x)=\frac{1}{2}n I_n(x)
  \quad\mbox{for }n\in\N \mbox{ and }x\in[-1,1]\;.
\end{equation}
An alternative derivation of~(\ref{I2jacobi}) which uses the
second order differential equations satisfied by
Jacobi polynomials is
given by Belinsky~\cite{belinsky}.
Moreover, according to~\cite[Theorem~3]{belinsky}, the polynomials
$\{I_n\colon n\in\N\}$ are
orthogonal on $[-1,1]$ with respect to the weighting function
$(1-x^2)^{-1}$. However, as this weighting function is not continuous
on $[-1,1]$ these polynomials do not belong to the class of classical
orthogonal polynomials.

To gain control over the Legendre polynomials and their integrals in the limit
$n\to\infty$, we
exploit an asymptotic property of Jacobi polynomials,
cf.~\cite[Theorem~8.21.8]{gabor}, which is due to Darboux~\cite{darboux}.
\begin{thm}[Darboux formula]\label{thm:darboux}
  Let $\alpha,\beta\in\R$ be arbitrary.
  For $\theta\in(0,\pi)$, set
  \begin{equation*}
    k(\theta)=\pi^{-\frac{1}{2}}
    \left(\sin\frac{\theta}{2}\right)^{-\alpha-\frac{1}{2}}
    \left(\cos\frac{\theta}{2}\right)^{-\beta-\frac{1}{2}}\;.
  \end{equation*}
  Then, as $n\to\infty$, we have
  \begin{equation*}
    P_n^{(\alpha,\beta)}\left(\cos\theta\right)=
    n^{-\frac{1}{2}}k(\theta)
    \cos\left(\left(n+\frac{\alpha+\beta+1}{2}\right)\theta
      -\left(\alpha+\frac{1}{2}\right)\frac{\pi}{2}\right)
    +O\left(n^{-\frac{3}{2}}\right)\;,
  \end{equation*}
  where the bound on the error term holds uniformly
  in $\theta\in[\eps,\pi-\eps]$ for $\eps>0$.
\end{thm}
By the Darboux formula, we particularly have, as $n\to\infty$,
\begin{align}
  P_n^{(0,0)}\left(\cos\theta\right)
  &=\sqrt{\frac{2}{n\pi\sin\theta}}\label{legasymp}
    \cos\left(\left(n+\frac{1}{2}\right)\theta
    -\frac{\pi}{4}\right)+O\left(n^{-\frac{3}{2}}\right)\;,
  \quad\mbox{and}\\
  P_n^{(-1,-1)}\left(\cos\theta\right)\label{intasymp}
  &=\sqrt{\frac{\sin\theta}{2n\pi}}
    \cos\left(\left(n-\frac{1}{2}\right)\theta
      +\frac{\pi}{4}\right)+O\left(n^{-\frac{3}{2}}\right)\;,
\end{align}
uniformly in $\theta\in[\eps,\pi-\eps]$ for $\eps>0$.
These asymptotics are used for estimates in
Section~\ref{fluctuations}.
\section{Iterated Kolmogorov loops}
\label{kol_loops}
We find two alternative representations for the iterated Kolmogorov
loop of step $N$ and we use the second representation,
cf.~Proposition~\ref{propn:kol_loop_leg}, to prove
Theorem~\ref{thm:LLN}. Whereas the first representation
is obtained by applying the most evident approach of
considering the first component of an
expression for the iterated
Kolmogorov diffusion $(\B_t^N)_{t\in[0,1]}$ of step $N$
conditioned on $\B_1^N=0$,
the second representation in terms of shifted Legendre polynomials
is much more useful for our analysis. This is due to
the orthogonality of the Legendre polynomials. Moreover, we see that
the first representation requires the inversion of a particular
$N\times N$ factorial Hankel matrix.

Throughout, for $l\in\{1,\dots,N\}$ and $t\in[0,1]$,
we write $\B_t^{N,l}$ to denote the $l^{\rm th}$ component of $\B_t^N$.
We observe that integration by parts yields
\begin{equation}\label{IBP}
  \B_t^{N,l}
  =\int_0^t\int_0^{s_{l-1}}\dots \int_0^{s_2} B_{s_1}\dd s_1\dots\dd s_{l-1}
  =\frac{1}{(l-1)!}\int_0^t(t-s)^{l-1}\dd B_s\;.
\end{equation}
To obtain the first representation for the iterated Kolmogorov
loop of step $N$, we follow a similar line of reasoning
as in~\cite[Section~4.4]{model_class}. 
\begin{propn}\label{propn:hankel_loop}
  Fix $N\in\N$. Let $\alpha_1,\dots,\alpha_N$ be the polynomials on
  $[0,1]$ given, for $t\in[0,1]$, by
  \begin{equation*}
    \alpha_l(t)=
    \sum_{k=1}^N(-1)^{N+k+l+1}(l-1)!\,\binom{N}{k}\binom{N+l-1}{l-1}
    \sum_{m=0}^{k-1}\binom{N-k+m}{l-1}\binom{N+m-1}{m} t^k\;.
  \end{equation*}
  Then the stochastic process $(Z_t^N)_{t\in[0,1]}$ in $\R$ defined by
  \begin{equation*}
    Z_t^N=B_t-\sum_{l=1}^N\alpha_l(t)\B_1^{N,l}
  \end{equation*}
  has the same law as the iterated Kolmogorov loop of step $N$.
\end{propn}
\begin{proof}
  Let $A$ be the $N\times N$ matrix and let $E$ be the $N\times 1$
  matrix with entries, for $k,l\in\{1,\dots,N\}$,
  \begin{equation*}
    A_{kl}=
    \begin{cases}
      1 & \mbox{if } k=l+1\\
      0 & \mbox{otherwise}
    \end{cases}
    \quad\mbox{and}\quad
    E_{k}=
    \begin{cases}
      1 & \mbox{if } k=1\\
      0 & \mbox{otherwise}
    \end{cases}\;.
  \end{equation*}
  Using the matrix exponential of a square matrix, we set, for
  $r\in[0,1]$,
  \begin{equation}\label{defn:U}
    U(r)=\e^{rA}E\;,
  \end{equation}
  and we define, for $t\in[0,1]$,
  \begin{equation}\label{defn:V}
    V(t)=\int_0^tU(t-s)U(-s)^T\dd s\;.
  \end{equation}
  Since
  \begin{equation}\label{exp:U}
    \left(U(r)\right)_k=\frac{r^{k-1}}{(k-1)!}\;,
  \end{equation}
  we compute with the help of~\cite[Lemma~3.2]{model_class} that
  \begin{equation*}
    \left(V(t)\right)_{kl}=\frac{1}{(k-1)!\,(l-1)!}
    \int_0^t(t-s)^{k-1}(-s)^{l-1}\dd s
    =(-1)^{l-1}\frac{t^{k+l-1}}{(k+l-1)!}\;.
  \end{equation*}
  We further observe that due to~(\ref{IBP}) and (\ref{exp:U})
  the iterated Kolmogorov
  diffusion of step $N$ can be expressed as
  \begin{equation}\label{itkolwithU}
    \B_t^N=\int_0^tU(t-s)\dd B_s\;.
  \end{equation}
  Let $({\boldsymbol{Z}}_t^N)_{t\in[0,1]}$ be the stochastic process
  in $\R^N$ given by
  \begin{equation}\label{split_exp}
    \B_t^N={\boldsymbol{Z}}_t^N+V(t)V(1)^{-1}\B_1^N\;.
  \end{equation}
  Using the expression~(\ref{itkolwithU}), applying the It\^o isometry
  and recalling the definitions~(\ref{defn:U}) and (\ref{defn:V}),
  we obtain
  \begin{equation*}
    \E\left[\B_t^{N}\left(\B_1^{N}\right)^T\right]
    =\int_0^tU(t-s)U(1-s)^T\dd s
    =V(t)\left(\e^A\right)^T\;.
  \end{equation*}
  It follows that
  \begin{equation*}
    \E\left[{\boldsymbol{Z}}_t^{N}\left(\B_1^{N}\right)^T\right]
    =V(t)\left(\e^A\right)^T
    -V(t)V(1)^{-1}V(1)\left(\e^A\right)^T=0
    \quad\mbox{for all }t\in[0,1]\;.
  \end{equation*}
  Since both the process $({\boldsymbol{Z}}_t^N)_{t\in[0,1]}$ and the
  random variable $\B_1^{N}$ have zero mean and are Gaussian,
  they are uncorrelated which implies that they are independent.
  Therefore, we deduce from~(\ref{split_exp}) that
  $({\boldsymbol{Z}}_t^N)_{t\in[0,1]}$ is the process obtained by
  conditioning $(\B_t^N)_{t\in[0,1]}$ on $\B_1^N=0$, and it suffices
  to show that the first component of $({\boldsymbol{Z}}_t^N)_{t\in[0,1]}$
  is $(Z_t^N)_{t\in[0,1]}$.
  This requires an explicit expression for the inverse
  $\left(V(1)\right)^{-1}$, which is easily derived from the formula
  given in~\cite{inverthankel}. We have
  \begin{equation*}
    \left(V(1)^{-1}\right)_{kl}
    =(-1)^{N+l}(k-1)!\,l!\,\binom{N-1}{k-1}\binom{N+l-1}{l}
    \sum_{m=0}^{k-1}\binom{N-k+m}{l-1}\binom{N+m-1}{m}\;,
  \end{equation*}
  and hence, for all $l\in\{1,\dots,N\}$ and all $t\in[0,1]$,
  we see that
  \begin{equation*}
    \left(V(t)V(1)^{-1}\right)_{1l}
    =\sum_{k=1}^N\left(V(t)\right)_{1k}\left(V(1)^{-1}\right)_{kl}
    =\alpha_l(t)\;.
  \end{equation*}
  Thus, we conclude
  \begin{equation*}
    {\boldsymbol{Z}}_t^{N,1}
    =\B_t^{N,1}-\sum_{l=1}^N\left(V(t)V(1)^{-1}\right)_{1l}\B_1^{N,l}
    =B_t-\sum_{l=1}^N\alpha_l(t)\B_1^{N,l}\;,
  \end{equation*}
  as needed.
\end{proof}
While the representation given in Proposition~\ref{propn:hankel_loop}
is obtained through a straightforward approach,
it appears to be too
complicated to proceed with, amongst others because the
polynomial coefficients $\alpha_1,\dots,\alpha_N$ in
front of the components of $\B_1^N$ change with $N$.
Instead, we use
the representation given in Proposition~\ref{propn:kol_loop_leg} for
our analysis.
\begin{proof}[Proof of Proposition~\ref{propn:kol_loop_leg}]
  Since the shifted Legendre polynomial $Q_n$ is a polynomial of
  degree $n$ which satisfies the parity relation~(\ref{parityQ}), it
  follows from~(\ref{IBP})
  that $\int_0^1Q_n(r)\dd B_r$ can be expressed as a linear combination of
  \begin{equation*}
    B_1,\int_0^1B_{s_1}\dd s_1,
    \int_0^1\int_0^{s_2}B_{s_1}\dd s_1\dd s_2,\dots,
    \int_0^1\int_0^{s_n}\dots \int_0^{s_2} B_{s_1}\dd s_1\dots\dd s_n\;.
  \end{equation*}
  Thus, for $N\in\N$ fixed, there exist polynomials
  $\beta_1,\dots,\beta_N$ on $[0,1]$ such that, for $t\in[0,1]$,
  \begin{equation*}
    \sum_{n=0}^{N-1}(2n+1)\int_0^tQ_n(r)\dd r \int_0^1Q_n(r)\dd B_r
    =\sum_{l=1}^N\beta_l(t)\B_1^{N,l}\;.
  \end{equation*}
  As the process $(L_t^N)_{t\in[0,1]}$ is defined,
  according to~(\ref{expr4L}), by
  \begin{equation*}
    L_t^N=
    B_t-\sum_{n=0}^{N-1}(2n+1)\int_0^tQ_n(r)\dd r \int_0^1Q_n(r)\dd B_r\;,
  \end{equation*}
  we can write
  \begin{equation}\label{bridge_split2}
    B_t=L_t^N+\sum_{l=1}^{N}\beta_l(t)\B_1^{N,l}\;.
  \end{equation}
  Using the It\^o isometry and the orthogonality of the shifted
  Legendre polynomials with~(\ref{orthconst4Q}),
  we obtain from~(\ref{expr4L}) that, for all $t\in[0,1]$ and all
  $m\in\{0,\dots,N-1\}$,
  \begin{equation*}
    \E\left[L_t^N\int_0^1Q_m(r)\dd B_r\right]=
    \int_0^tQ_m(r)\dd r-
    \sum_{n=0}^{N-1}(2n+1)\int_0^tQ_n(r)\dd r \int_0^1Q_n(r)Q_m(r)\dd r=0\;.
  \end{equation*}
  By the completeness of the shifted Legendre polynomials and the
  identity~(\ref{IBP}), this implies that,
  for all $t\in[0,1]$ and all $l\in\{1,\dots,N\}$,
  \begin{equation*}
    \E\left[L_t^N\B_1^{N,l}\right]=0\;.
  \end{equation*}
  Hence, $L_t^N$ and $\B_1^N$ are uncorrelated for all $t\in[0,1]$,
  which due to $(L_t^N)_{t\in[0,1]}$ and $\B_1^N$ both being Gaussian
  shows that $(L_t^N)_{t\in[0,1]}$ and $\B_1^N$ are independent.
  From the representation~(\ref{bridge_split2}), we finally deduce
  that $(L_t^N)_{t\in[0,1]}$ is indeed equal in law to
  the iterated Kolmogorov loop of step $N$.
\end{proof}
The advantage of the representation for the iterated Kolmogorov loops
given in Proposition~\ref{propn:kol_loop_leg} 
over the one given in Proposition~\ref{propn:hankel_loop} is that
by the orthogonality of the shifted Legendre polynomials,
it gives rise to 
a neat expression for the covariance functions of the iterated
Kolmogorov loops.
\begin{lemma}\label{lem:covariance}
  The iterated Kolmogorov loop of step $N\in\N$ is a zero-mean Gaussian
  process with covariance $C_N$ given, for $s,t\in[0,1]$, by
  \begin{equation*}
    C_N(s,t)=
    \min(s,t)-\sum_{n=0}^{N-1}(2n+1)\int_0^sQ_n(r)\dd r\int_0^tQ_n(r)\dd r\;.
  \end{equation*}
\end{lemma}
\begin{proof}
  By Proposition~\ref{propn:kol_loop_leg}, it suffices to show that
  $(L_t^N)_{t\in[0,1]}$ is a zero-mean Gaussian process with the
  specified covariance function. From the definition~(\ref{expr4L}),
  we see that $(L_t^N)_{t\in[0,1]}$ is a zero-mean
  Gaussian process. Regarding its covariance function,
  the It\^o isometry and the orthogonality of the shifted Legendre
  polynomials with~(\ref{orthconst4Q})
  imply that, for $s,t\in[0,1]$,
  \begin{equation*}
    \E\left[B_s\sum_{n=0}^{N-1}(2n+1)
      \int_0^tQ_n(r)\dd r \int_0^1Q_n(r)\dd B_r\right]
    =\sum_{n=0}^{N-1}(2n+1)\int_0^sQ_n(r)\dd r \int_0^tQ_n(r)\dd r
  \end{equation*}
  as well as
  \begin{align*}
    &\E\left[\left(\sum_{n=0}^{N-1}
        (2n+1)\int_0^sQ_n(r)\dd r \int_0^1Q_n(r)\dd B_r\right)
      \left(\sum_{n=0}^{N-1}
        (2n+1)\int_0^tQ_n(r)\dd r \int_0^1Q_n(r)\dd B_r\right)\right]\\
    &\qquad=\sum_{n=0}^{N-1}
    (2n+1)\int_0^sQ_n(r)\dd r \int_0^tQ_n(r)\dd r\;,
  \end{align*}
  which together with $\E[B_sB_t]=\min(s,t)$ yields
  \begin{equation*}
    C_N(s,t)=\E\left[L_s^NL_t^N\right]
    =\min(s,t)-\sum_{n=0}^{N-1}(2n+1)\int_0^sQ_n(r)\dd r\int_0^tQ_n(r)\dd r\;,
  \end{equation*}
  as required.
\end{proof}
This characterisation of the iterated Kolmogorov loops allows us to
prove Theorem~\ref{thm:LLN} by exploiting the
completeness and orthogonality
of the shifted Legendre polynomials. The argument follows
a line of reasoning which is part of the usual proof of Mercer's
theorem, see~\cite[Part~IV]{mercer}.
\begin{proof}[Proof of Theorem~\ref{thm:LLN}]
  From the discussion in Section~\ref{legendre}, we recall that
  $\{\sqrt{2n+1}Q_n\colon n\in\N_0\}$ forms a complete orthonormal set
  of polynomials in $L^2[0,1]$ with respect to the usual inner
  product. As a consequence, the polarised Parseval identity applies
  to give, for $s,t\in[0,1]$,
  \begin{align}
    \begin{aligned}\label{id:parseval}
      \min(s,t)=\int_0^1\ind_{[0,s]}(r)\ind_{[0,t]}(r)\dd r
      &=\sum_{n=0}^\infty(2n+1)\int_0^1\ind_{[0,s]}(r)Q_n(r)\dd r
      \int_0^1\ind_{[0,t]}(r)Q_n(r)\dd r\\
      &=\sum_{n=0}^\infty(2n+1)\int_0^s Q_n(r)\dd r \int_0^t Q_n(r)\dd r\;.
    \end{aligned}
  \end{align}
  Due to Lemma~\ref{lem:covariance}, it follows that the
  covariance $C_N$ of the iterated Kolmogorov loop of step $N$ is
  given, for $s,t\in[0,1]$, by
  \begin{equation*}
    C_N(s,t)
    =\sum_{n=N}^\infty(2n+1)\int_0^s Q_n(r)\dd r \int_0^t Q_n(r)\dd r\;.
  \end{equation*}
  Using Cauchy-Schwarz, we obtain that, for all $N,M\in\N$ with $N<M$
  and for $s,t\in[0,1]$ fixed,
  \begin{align*}
    &\left|\sum_{n=N}^M(2n+1)\left|\int_0^s Q_n(r)\dd r\right|
      \left|\int_0^t Q_n(r)\dd r\right|\right|^2\\
    &\qquad\leq\sum_{n=N}^M(2n+1)\left(\int_0^s Q_n(r)\dd r\right)^2
    \sum_{n=N}^M(2n+1)\left(\int_0^t Q_n(r)\dd r\right)^2\;,
  \end{align*}
  and thus, by~(\ref{id:parseval}), we have
  \begin{equation*}
    \sum_{n=N}^M(2n+1)\left|\int_0^s Q_n(r)\dd r\right|
    \left|\int_0^t Q_n(r)\dd r\right|\leq\sqrt{st}\;.
  \end{equation*}
  This implies that the series representation for $\min(s,t)$
  in~(\ref{id:parseval}) converges absolutely. In particular, the
  sequence $(G_N)_{N\in\N}$ of functions
  $G_N\colon[0,1]\times[0,1]\to\R$ defined by
  \begin{equation*}
    G_N(s,t)=\sum_{n=N}^\infty(2n+1)\left|\int_0^s Q_n(r)\dd r\right|
    \left|\int_0^t Q_n(r)\dd r\right| 
  \end{equation*}
  converges pointwise to zero as $N\to\infty$. As $(G_N)_{N\in\N}$ is
  a monotonically decreasing sequence of continuous real-valued
  functions, Dini's theorem applies to give that
  $(G_N)_{N\in\N}$ converges uniformly on
  the compact set $[0,1]\times[0,1]$ to 
  the zero function. By the Cauchy criterion, we further deduce that
  the sequence $(C_N)_{N\in\N}$ of covariances
  converges uniformly on $[0,1]\times[0,1]$ to 
  the zero function. As the iterated Kolmogorov loops are
  zero-mean Gaussian processes and since their covariance functions converge
  uniformly as $N\to\infty$ to zero, it follows, e.g.
  by~\cite[Section~3]{gaussian}, that the iterated
  Kolmogorov loops of step $N$ indeed converge weakly as $N\to\infty$
  to the zero process on $\Omega^{0,0}$.
\end{proof}
Note that Proposition~\ref{propn:kol_loop_leg} and
Theorem~\ref{thm:LLN} together show that Brownian motion $(B_t)_{t\in[0,1]}$
admits the decomposition
\begin{equation*}
  \left(\sum_{n=0}^\infty(2n+1)\int_0^tQ_n(r)\dd r
    \int_0^1Q_n(r)\dd B_r\right)_{t\in[0,1]}\;,
\end{equation*}
which differs from the usual Karhunen--Lo{\`e}ve expansion,
cf.~\cite[page~144]{loeve2}, for Brownian motion,
and which alternatively could be expressed
in terms of the representation given in
Proposition~\ref{propn:hankel_loop}.
Foster, Lyons and Oberhauser~\cite{foster}
independently obtained this decomposition with the difference that
the random coefficients of the integrals of the shifted
Legendre polynomials are defined using
the Brownian bridge process
associated with $(B_t)_{t\in[0,1]}$. They use this representation to
generate approximate sample paths of Brownian motion which respect
integration of polynomials up to a fixed degree.
\section{Moment analysis on the diagonal}
\label{moments}
As the first step towards proving Theorem~\ref{thm:CLT4P}, we
establish the convergence of moments on the diagonal.
Throughout, we use the families of complex-valued polynomials
introduced in Section~\ref{complexext} and Section~\ref{general_int}
to simplify the presentation of our analysis. We repeatedly
expand terms into their partial fraction decomposition because this
reveals that certain sums we encounter telescope.
Let $S_N\colon[-1,1]\to\R$ be the
restriction of $R_N$ to the diagonal, that is,
\begin{equation*}
  S_N(x)=R_N(x,x)
  \quad\mbox{for }x\in[-1,1]\;.
\end{equation*}
Due to~(\ref{Iisint}), we can write
\begin{equation*}
  S_N(x)=N\left(1+x-\frac{1}{2}\left((1+x)^2
    +\sum_{n=1}^{N-1}(2n+1)\left(I_n(x)\right)^2\right)\right)
  \quad\mbox{for }x\in[-1,1]\;.
\end{equation*}
To study the moments of $S_N$ in the limit $N\to\infty$, we start by
considering each summand separately. In particular,
for all $k\in\N_0$, we have
\begin{equation}\label{moment0}
  \int_{-1}^1x^{2k}(1+x)\dd x=\frac{2}{2k+1}
  \qquad\mbox{and}\qquad
  \int_{-1}^1x^{2k+1}(1+x)\dd x=\frac{2}{2k+3}
\end{equation}
as well as
\begin{equation}\label{moment1}
  \frac{1}{2}\int_{-1}^1x^{2k}(1+x)^2\dd x
  =\frac{1}{2k+1}+\frac{1}{2k+3}
  \qquad\mbox{and}\qquad
  \frac{1}{2}\int_{-1}^1x^{2k+1}(1+x)^2\dd x=\frac{2}{2k+3}\;.
\end{equation}
The remaining odd moments all vanish.
\begin{lemma}\label{lem:oddmom}
  For all $n\in\N$ and all $k\in\N_0$, we have
  \begin{equation*}
    \int_{-1}^1x^{2k+1}\left(I_n(x)\right)^2\dd x=0\;.
  \end{equation*}
\end{lemma}
\begin{proof}
  By the parity property~(\ref{parityI}), we know that $I_n^2$ is
  an even function on $[-1,1]$ for all $n\in\N$. Therefore, the integrand
  of the above integral is an odd function on $[-1,1]$, and
  it follows that the integral vanishes.
\end{proof}
We are left with studying the remaining even moments, which
is the core of our moment analysis. The recursive method we develop
requires us to look at additional moments to the ones we would like to
consider.
For all $p,q\in\Z$ and $k\in\N_0$, we set
\begin{equation*}
  m_{p,q}^k=(p+q+1)\int_{-1}^1x^{2k}I_p(x)I_q(x)\dd x\;.
\end{equation*}
These moments satisfy the following recursion formula.
This is the first time where the extension of the family of the
integrals of the Legendre polynomials comes in handy as we do not have to
deal with a boundary at $n=0$.
\begin{lemma}\label{lem:mom_rec}
  For all $p,q\in\Z$ and for $k\in\N$, we have
  \begin{align*}
    m_{p,q}^k
    &=\frac{(p+q+1)(p+2)(q+2)}{(2p+1)(2q+1)(p+q+3)}m_{p+1,q+1}^{k-1}
     +\frac{(p+q+1)(p-1)(q-1)}{(2p+1)(2q+1)(p+q-1)}m_{p-1,q-1}^{k-1}\\
    &\qquad+\frac{(p+2)(q-1)}{(2p+1)(2q+1)}m_{p+1,q-1}^{k-1}
      +\frac{(p-1)(q+2)}{(2p+1)(2q+1)}m_{p-1,q+1}^{k-1}\;.
  \end{align*}
\end{lemma}
\begin{proof}
  According to Lemma~\ref{lem:rec4I}, we have both
  \begin{align*}
    (p+2)I_{p+1}(x)&=(2p+1)xI_p(x)-(p-1)I_{p-1}(x)
    \quad\mbox{for }x\in[-1,1]\\
  \intertext{and}
    (q+2)I_{q+1}(x)&=(2q+1)xI_q(x)-(q-1)I_{q-1}(x)
    \quad\mbox{for }x\in[-1,1]\;.
  \end{align*}
  It follows that
  \begin{align*}
    &(2p+1)(2q+1)\int_{-1}^1x^{2k}I_p(x)I_q(x)\dd x\\
    &\qquad=\int_{-1}^1x^{2k-2}
      \left((p+2)I_{p+1}(x)+(p-1)I_{p-1}(x)\right)
      \left((q+2)I_{q+1}(x)+(q-1)I_{q-1}(x)\right)\dd x\;,
  \end{align*}
  which yields the desired result.
\end{proof}
Moreover, we have the partial fraction decompositions specified below.
\begin{propn}\label{propn:PFD}
  There exists a family
  $\{b_{a,k}^l\in\R\colon a\in\Z\mbox{ and }k,l\in\N_0\}$ of
  coefficients satisfying
  \begin{equation}\label{PFD:sym}
    b_{a,k}^l=b_{-a,k}^l
    \quad\mbox{for all }a\in\Z\mbox{ and }k,l\in\N_0
  \end{equation}
  as well as 
  \begin{equation}\label{PFD:bound}
    b_{0,k}^{a-1}+2a\sum_{l=0}^k\frac{b_{a,k}^l}{l+1}=0
    \quad\mbox{and}\quad
    \frac{b_{c,k}^{a-1}}{a}=\frac{b_{a,k}^{c-1}}{c}
    \quad\mbox{for all }a,c\in\N\mbox{ and }k\in\N_0\;,
  \end{equation}
  such that, for all $n,a\in\Z$ and all $k\in\N_0$,
  \begin{equation}\label{PFD}
    m_{n-a,n+a}^k=\sum_{l=0}^k\frac{b_{a,k}^l}{2n-2l-1}-
    \sum_{l=0}^k\frac{b_{a,k}^l}{2n+2l+3}\;,
  \end{equation}
  where it is understood that $b_{a,k}^l=0$ if $l>k$.
\end{propn}
\begin{proof}
  The proof works by induction on $k\in\N_0$. For the base case,
  we start by observing that the orthogonality of the Legendre polynomials
  and the definition~(\ref{defn:complexL}) imply that
  \begin{equation}\label{genorth4P}
    \int_{-1}^1P_n(x)P_m(x)\dd x = 0
    \quad\mbox{if }n\not=m \mbox{ and }n\not= -m-1\;.
  \end{equation}
  In particular, the integral vanishes if $n\not=m$ but $n+m$ is even.
  Using~(\ref{defn:I}) and (\ref{scale4Z}) we compute that,
  for all $n\in\Z$,
  \begin{align*}
    m_{n,n}^0=(2n+1)\int_{-1}^1\left(I_n(x)\right)^2\dd x
    &=\frac{1}{2n+1}\int_{-1}^1\left(P_{n+1}(x)-P_{n-1}(x)\right)^2\dd x\\
    &=\frac{1}{2n+1}\left(\frac{2}{2n-1}+\frac{2}{2n+3}\right)
      =\frac{1}{2n-1}-\frac{1}{2n+3}\;,
  \end{align*}
  and similarly,
  \begin{equation*}
    m_{n-1,n+1}^0=m_{n+1,n-1}^0
    =(2n+1)\int_{-1}^1I_{n-1}(x)I_{n+1}(x)\dd x
    =-\frac{1}{2}\left(\frac{1}{2n-1}-\frac{1}{2n+3}\right)
  \end{equation*}
  as well as
  \begin{equation*}
    m_{n-a,n+a}^0=0
    \quad\mbox{for all }a\in\Z\setminus\{-1,0,1\}\;.
  \end{equation*}
  Hence, for $k=0$, the moments are indeed of the form~(\ref{PFD})
  with the only non-zero coefficients
  \begin{equation}\label{coeff4k0}
    b_{0,0}^0=1
    \quad\mbox{and}\quad
    b_{1,0}^0=b_{-1,0}^0=-\frac{1}{2}\;.
  \end{equation}
  In particular, we have $b_{0,0}^0+2b_{1,0}^0=0$ and the
  relations~(\ref{PFD:sym}) as well as (\ref{PFD:bound}) are satisfied
  for $k=0$, which settles the base case.
  For the induction step, we start with $a=0$. Applying
  Lemma~\ref{lem:mom_rec} and the induction hypothesis yields
  \begin{align}\label{PFD_mom_rec}
    \begin{aligned}
      m_{n,n}^k
      &=\frac{(n+2)^2}{(2n+1)(2n+3)}
      \left(\sum_{l=0}^{k-1}\frac{b_{0,k-1}^l}{2n-2l+1}-
        \sum_{l=0}^{k-1}\frac{b_{0,k-1}^l}{2n+2l+5}\right)\\
      &\qquad+\frac{(n-1)^2}{(2n-1)(2n+1)}
      \left(\sum_{l=0}^{k-1}\frac{b_{0,k-1}^l}{2n-2l-3}-
        \sum_{l=0}^{k-1}\frac{b_{0,k-1}^l}{2n+2l+1}\right)\\
      &\qquad+\frac{2(n-1)(n+2)}{(2n+1)^2}
      \left(\sum_{l=0}^{k-1}\frac{b_{1,k-1}^l}{2n-2l-1}-
        \sum_{l=0}^{k-1}\frac{b_{1,k-1}^l}{2n+2l+3}\right)\;.
    \end{aligned}
  \end{align}
  Through this expression, we can define what it means to evaluate
  $(2n+1)^2m_{n,n}^k$ at $n=-1/2$. By additionally
  using the relation~(\ref{PFD:bound}) of the induction hypothesis, we
  obtain that
  \begin{equation}\label{term_vanishes}
    \left.(2n+1)^2m_{n,n}^k\right|_{n=-1/2}
    =\frac{9}{4}b_{0,k-1}^0+\frac{9}{2}\sum_{l=0}^{k-1}\frac{b_{1,k-1}^l}{l+1}
    =\frac{9}{4}\left(b_{0,k-1}^0+2\sum_{l=0}^{k-1}\frac{b_{1,k-1}^l}{l+1}\right)
    =0\;.
  \end{equation}
  Since the remaining factors in the denominators of the terms giving
  $m_{n,n}^k$ only ever appear linearly, we deduce
  from~(\ref{PFD_mom_rec}) as well as
  (\ref{term_vanishes}), and by referring to
  the Heaviside cover-up method that $m_{n,n}^k$ is of the form
  \begin{equation}\label{basic_PFD}
    m_{n,n}^k=\sum_{l=0}^k\frac{b_{0,k}^l}{2n-2l-1}+
    \sum_{l=0}^k\frac{c_{0,k}^l}{2n+2l+3}+\frac{d_{0,k}}{2n+1}\;,
  \end{equation}
  for suitable coefficients
  $b_{0,k}^l,c_{0,k}^l,d_{0,k}\in\R$. By Lemma~\ref{Isym}, we
  further have
  \begin{equation}\label{mom_sym}
    m_{-n-1,-n-1}^k
    =-(2n+1)\int_{-1}^1x^{2k}\left(I_{-n-1}(x)\right)^2\dd x
    =m_{n,n}^k
    \quad\mbox{for all }n\in\N\;.
  \end{equation}
  However, according to~(\ref{basic_PFD}), we can write
  \begin{equation*}
    m_{-n-1,-n-1}^k=-\sum_{l=0}^k\frac{c_{0,k}^l}{2n-2l-1}
    -\sum_{l=0}^k\frac{b_{0,k}^l}{2n+2l+3}-\frac{d_{0,k}}{2n+1}\;,
  \end{equation*}
  and as a consequence of~(\ref{mom_sym}), it follows that
  \begin{equation*}
    d_{0,k}=0
    \quad\mbox{and}\quad
    c_{0,k}^l=-b_{0,k}^l
    \quad\mbox{for }l\in\{0,\dots,k\}\;,
  \end{equation*}
  which gives the desired form~(\ref{PFD}) for $a=0$.
  Similarly, for $a=\pm 1$, we use
  \begin{align*}
    \begin{aligned}
      m_{n-1,n+1}^k=m_{n+1,n-1}^k
      &=\frac{(2n+1)(n+1)(n+3)}{(2n-1)(2n+3)^2}
      \left(\sum_{l=0}^{k-1}\frac{b_{1,k-1}^l}{2n-2l+1}-
        \sum_{l=0}^{k-1}\frac{b_{1,k-1}^l}{2n+2l+5}\right)\\
      &\qquad+
      \frac{(2n+1)(n-2)n}{(2n-1)^2(2n+3)}
      \left(\sum_{l=0}^{k-1}\frac{b_{1,k-1}^l}{2n-2l-3}-
        \sum_{l=0}^{k-1}\frac{b_{1,k-1}^l}{2n+2l+1}\right)\\
      &\qquad+
      \frac{n(n+1)}{(2n-1)(2n+3)}
      \left(\sum_{l=0}^{k-1}\frac{b_{0,k-1}^l}{2n-2l-1}-
        \sum_{l=0}^{k-1}\frac{b_{0,k-1}^l}{2n+2l+3}\right)\\
      &\qquad+
      \frac{(n-2)(n+3)}{(2n-1)(2n+3)}
      \left(\sum_{l=0}^{k-1}\frac{b_{2,k-1}^l}{2n-2l-1}-
        \sum_{l=0}^{k-1}\frac{b_{2,k-1}^l}{2n+2l+3}\right)
    \end{aligned}
  \end{align*}
  to give a meaning to
  \begin{align*}
    \left.(2n-1)^2m_{n-1,n+1}^k\right|_{n=1/2}
    &=\left.(2n+3)^2m_{n-1,n+1}^k\right|_{n=-3/2}\\
    &=\frac{3}{16}
      \left(b_{0,k-1}^0+2\sum_{l=0}^{k-1}\frac{b_{1,k-1}^l}{l+1}\right)
      +\frac{21}{16}\left(\frac{b_{1,k-1}^1}{2}-b_{2,k-1}^0\right)\;,
  \end{align*}
  whereas, for $a\in\Z\setminus\{-1,0,1\}$, we have
  \begin{align*}
      m_{n-a,n+a}^{k}
      &=\frac{(2n+1)(n-a+2)(n+a+2)}{(2n-2a+1)(2n+2a+1)(2n+3)}
      \left(\sum_{l=0}^{k-1}\frac{b_{a,k-1}^l}{2n-2l+1}-
        \sum_{l=0}^{k-1}\frac{b_{a,k-1}^l}{2n+2l+5}\right)\\
      &\qquad+
      \frac{(2n+1)(n-a-1)(n+a-1)}{(2n-2a+1)(2n+2a+1)(2n-1)}
      \left(\sum_{l=0}^{k-1}\frac{b_{a,k-1}^l}{2n-2l-3}-
        \sum_{l=0}^{k-1}\frac{b_{a,k-1}^l}{2n+2l+1}\right)\\
      &\qquad+
      \frac{(n-a+2)(n+a-1)}{(2n-2a+1)(2n+2a+1)}
      \left(\sum_{l=0}^{k-1}\frac{b_{a-1,k-1}^l}{2n-2l-1}-
        \sum_{l=0}^{k-1}\frac{b_{a-1,k-1}^l}{2n+2l+3}\right)\\
      &\qquad+
      \frac{(n-a-1)(n+a+2)}{(2n-2a+1)(2n+2a+1)}
      \left(\sum_{l=0}^{k-1}\frac{b_{a+1,k-1}^l}{2n-2l-1}-
        \sum_{l=0}^{k-1}\frac{b_{a+1,k-1}^l}{2n+2l+3}\right)
  \end{align*}
  to make sense of
  \begin{align*}
    &\left.(2n-2a+1)^2m_{n-a,n+a}^{k}\right|_{n=a-1/2}
    =\left.(2n+2a+1)^2m_{n-a,n+a}^{k}\right|_{n=-a-1/2}\\
    &\qquad =\frac{3(4a+3)}{16}
      \left(\frac{b_{a,k-1}^a}{a+1}-\frac{b_{a+1,k-1}^{a-1}}{a}\right)
      +\frac{3(4a-3)}{16}
      \left(\frac{b_{a-1,k-1}^{a-1}}{a}-\frac{b_{a,k-1}^{a-2}}{a-1}\right)\;.
  \end{align*}
  By the
  relation~(\ref{PFD:bound}) of the induction hypothesis, it follows that,
  for all $a\in\Z$,
  \begin{equation}\label{no_blow_up}
    \left.(2n-2a+1)^2m_{n-a,n+a}^{k}\right|_{n=a-1/2}
    =\left.(2n+2a+1)^2m_{n-a,n+a}^{k}\right|_{n=-a-1/2}=0\;.
  \end{equation}
  Hence, as before, we use the expression for the moments, the
  Heaviside cover-up method and the symmetry property
  \begin{equation*}
    m_{-n-1-a,-n-1+a}^k
    =-(2n+1)\int_{-1}^1x^{2k}I_{-(n+a)-1}(x)I_{-(n-a)-1}(x)\dd x
    =m_{n-a,n+a}^k
    \quad\mbox{for } n\geq a+1
  \end{equation*}
  to deduce that we indeed have the partial fraction
  decomposition~(\ref{PFD}). The symmetry relation~(\ref{PFD:sym})
  is satisfied since
  \begin{equation*}
    m_{n-a,n+a}^{k}=m_{n+a,n-a}^{k}
    \quad\mbox{for all }n,a\in\Z\;.
  \end{equation*}
  To conclude the
  induction step, we still need to show that (\ref{PFD:bound}) holds.
  As we have just established that the moments are of the
  form~(\ref{PFD}), we are
  justified to define, for $a,c\in\N_0$,
  \begin{equation*}
    d_{a,k}^c=\left.\frac{(2n-2c+1)m_{n-a,n+a}^k}{2n+1}\right|_{n=c-1/2},
  \end{equation*}
  where it is understood that
  \begin{equation*}
    d_{a,k}^0=\left.m_{n-a,n+a}^k\right|_{n=-1/2}
    =-\sum_{l=0}^k\frac{b_{a,k}^l}{l+1}\;,
  \end{equation*}
  and where, for $c\in\N$, we have
  \begin{equation*}
    d_{a,k}^c=\frac{b_{a,k}^{c-1}}{2c}\;.
  \end{equation*}
  Thus, the relation~(\ref{PFD:bound}) of the induction hypothesis
  tells us that
  \begin{equation}\label{re_indhyp}
    d_{c,k-1}^a=d_{a,k-1}^c
    \quad\mbox{for all }a,c\in\N_0\;.
  \end{equation}
  Using Lemma~\ref{lem:mom_rec}, we obtain
  \begin{align}\label{boundrel2}
    \begin{aligned}
      d_{c,k}^a
      &=\frac{(2a-2c+3)(2a+2c+3)}{16(a-c)(a+c)}d_{c,k-1}^{a+1}
      +\frac{(2a-2c-3)(2a+2c-3)}{16(a-c)(a+c)}d_{c,k-1}^{a-1}\\
      &\qquad+
      \frac{(2a-2c+3)(2a+2c-3)}{16(a-c)(a+c)}d_{c-1,k-1}^a
      +\frac{(2a-2c-3)(2a+2c+3)}{16(a-c)(a+c)}d_{c+1,k-1}^a
    \end{aligned}
  \end{align}
  as well as
  \begin{align}\label{boundrel1}
    \begin{aligned}
      d_{a,k}^c
      &=\frac{(2c-2a+3)(2c+2a+3)}{16(c-a)(c+a)}d_{a,k-1}^{c+1}
      +\frac{(2c-2a-3)(2c+2a-3)}{16(c-a)(c+a)}d_{a,k-1}^{c-1}\\
      &\qquad+
      \frac{(2c-2a+3)(2c+2a-3)}{16(c-a)(c+a)}d_{a-1,k-1}^c
      +\frac{(2c-2a-3)(2c+2a+3)}{16(c-a)(c+a)}d_{a+1,k-1}^c\;.
    \end{aligned}
  \end{align}
  Due to~(\ref{re_indhyp}), the first summand on the right
  hand side of~(\ref{boundrel2}) agrees with the fourth summand on the
  right hand side of~(\ref{boundrel1}). Similarly, the second summand
  in~(\ref{boundrel2}) coincides with the third summand
  in~(\ref{boundrel1}). As the remaining terms also match, we
  see that
  \begin{equation*}
    d_{c,k}^a=d_{a,k}^c
    \quad\mbox{for all }a,c\in\N_0\;,
  \end{equation*}
  which implies the relation~(\ref{PFD:bound}) and concludes the
  proof.
\end{proof}
By a more in-depth analysis than the one performed in the proof of
Proposition~\ref{propn:PFD}, it is possible to use the Heaviside
cover-up method to obtain recurrence relations for the coefficients
$b_{a,k}^l$ which characterise them uniquely. However, as it is not
necessary for our subsequent analysis to determine each
coefficient $b_{a,k}^l$
separately, we postpone the derivation of recursion formulae to
the Appendix.
In the following, we see that to study the moments of $S_N$
in the limit
$N\to\infty$ it suffices to gain control over, for
$a\in\Z$ and $k\in\N_0$,
\begin{equation*}
  B_{a,k}=\sum_{l=0}^k(l+1)b_{a,k}^l\;.
\end{equation*}

These sums satisfy a much simpler recurrence relation than the
coefficients $b_{a,k}^l$ themselves, where
$B_{a,k}=B_{-a,k}$ as a result of
the symmetry property~(\ref{PFD:sym}).
\begin{propn}\label{propn:rec_rel4B}
  For all $k\in\N$ and all $a\in\Z$, we have
  \begin{equation}\label{rec4B}
    B_{a,k}=\frac{1}{4}B_{a-1,k-1}+\frac{1}{2}B_{a,k-1}+\frac{1}{4}B_{a+1,k-1}\;.
  \end{equation}
\end{propn}
\begin{proof}
  Using the partial fraction decomposition~(\ref{PFD}) of
  Proposition~\ref{propn:PFD} and
  \begin{equation*}
    \frac{1}{2n-2l-1}-\frac{1}{2n+2l+3}
    =\frac{4(l+1)}{(2n-2l-1)(2n+2l+3)}\;,
  \end{equation*}
  we deduce that
  \begin{equation}\label{sumconv1}
    \lim_{n\to\infty}n^2m_{n-a,n+a}^k
    =\lim_{n\to\infty}\sum_{l=0}^k\frac{4n^2(l+1)}{(2n-2l-1)(2n+2l+3)}b_{a,k}^l
    =B_{a,k}\;,
  \end{equation}
  and similarly,
  \begin{equation}\label{sumconv2}
    \lim_{n\to\infty}(n-1)^2m_{n-a,n+a}^k
    =\lim_{n\to\infty}(n+1)^2m_{n-a,n+a}^k=B_{a,k}\;.
  \end{equation}
  On the other hand, by applying Lemma~\ref{lem:mom_rec} with $p=n-a$
  and $q=n+a$, we obtain that
  \begin{align*}
    \lim_{n\to\infty}n^2 m_{n-a,n+a}^k
    &=\frac{1}{4}\lim_{n\to\infty}n^2 m_{n-a+1,n+a+1}^{k-1}
     +\frac{1}{4}\lim_{n\to\infty}n^2 m_{n-a-1,n+a-1}^{k-1}\\
    &\qquad+\frac{1}{4}\lim_{n\to\infty}n^2 m_{n-a+1,n+a-1}^{k-1}
      +\frac{1}{4}\lim_{n\to\infty}n^2 m_{n-a-1,n+a+1}^{k-1}\;,
  \end{align*}
  which together with (\ref{sumconv1}) and (\ref{sumconv2}) implies
  the claimed recurrence relation.
\end{proof}
\begin{remark}\label{rem:catalan}
  The recurrence relation~(\ref{rec4B})
  in Proposition~\ref{propn:rec_rel4B} can be
  rewritten as
  \begin{equation*}
    4^kB_{a,k}
    =4^{k-1}B_{a-1,k-1}+2\left(4^{k-1}B_{a,k-1}\right)+4^{k-1}B_{a+1,k-1}
    \quad\mbox{for }k\in\N\mbox{ and }a\in\Z\;,
  \end{equation*}
  that is, the numbers $4^kB_{a,k}$ satisfy the same recurrence
  relation as elements of the Catalan triangle which Shapiro
  introduced in~\cite{shapiro} and as elements of other Catalan
  triangles, e.g. see~\cite{catalan1,catalan2}.
\end{remark}
By the preceding remark, it should not come as a surprise that we
encounter the Catalan numbers when determining the sums
$B_{a,k}$ explicitly.
For $k\in\N_0$, the $k^{\rm th}$ Catalan number $C_k$ is given by
\begin{equation*}
  C_k=\frac{1}{k+1}\binom{2k}{k}
  =\binom{2k}{k}-\binom{2k}{k+1}\;.
\end{equation*}
In the next lemma, it is understood that
\begin{equation*}
  \binom{k}{l}=0
  \quad\mbox{if }k,l\in\N_0\mbox{ with }k<l\;.
\end{equation*}
\begin{lemma}\label{lem:catalan}
  We have $B_{0,0}=1$ and, for $a,k\in\N_0$ with $a+k\geq 1$,
  \begin{equation}\label{exp4B}
    B_{-a,k}=B_{a,k}=4^{-k}\left[\binom{2k}{k+a}-\frac{1}{2}
      \left[\binom{2k}{k+a-1}+\binom{2k}{k+a+1}\right]\right]\;.
  \end{equation}
  In particular, we see that
  \begin{equation}\label{B2catalan}
    B_{0,k}=4^{-k}C_k
    \quad\mbox{for all }k\in\N_0\;.
  \end{equation}
\end{lemma}
\begin{proof}
  By the recursion formula in Lemma~\ref{lem:rec4I}
  and the definition~(\ref{defn:I}), the polynomial $x^kI_{n-a}(x)$ is
  a linear combination of the polynomials
  \begin{equation}\label{indices1}
    P_{n-a-k-1}(x),P_{n-a-k+1}(x),\dots,P_{n-a+k-1}(x),P_{n-a+k+1}(x)\;,
  \end{equation}
  and similarly, $x^kI_{n+a}(x)$ is a linear combination of
  \begin{equation}\label{indices2}
    P_{n+a-k-1}(x),P_{n+a-k+1}(x),\dots,P_{n+a+k-1}(x),P_{n+a+k+1}(x)\;.
  \end{equation}
  Hence, if $n-a+k+1<n+a-k-1$, that is, if $k<a-1$, it follows
  from~(\ref{genorth4P}) that
  \begin{equation*}
    m_{n-a,n+a}^k=(2n+1)\int_{-1}^1x^{2k}I_{n-a}(x)I_{n+a}(x)\dd x=0
  \end{equation*}
  because all indices in~(\ref{indices1}) and (\ref{indices2})
  have the same parity. We deduce that
  \begin{equation*}
    B_{-a,k}=B_{a,k}=0
    \quad\mbox{if }k\leq a-2\;,
  \end{equation*}
  which is consistent with~(\ref{exp4B}). We further obtain
  from~(\ref{coeff4k0}) that
  \begin{equation*}
    B_{0,0}=1
    \quad\mbox{and}\quad
    B_{-1,0}=B_{1,0}=-\frac{1}{2}\;,
  \end{equation*}
  as claimed. Since this fixes the
  boundary values of our recursion, it suffices to verify
  that~(\ref{exp4B}) satisfies the recurrence
  relation~(\ref{rec4B}). This can be done by observing that the
  combinatorial numbers $C_{m,l}$ defined,
  for $m\in\N$ and $l\in\N_0$, by
  \begin{equation*}
    C_{m,l}=\frac{m-2l}{m}\binom{m}{l}
  \end{equation*}
  satisfy the recurrence relation
  \begin{equation}\label{rec4C}
    C_{m+2,l+1}=C_{m,l-1}+2 C_{m,l}+C_{m,l+1}
    \quad\mbox{for }m,l\in\N\;,
  \end{equation}
  see~\cite[Proposition~2.1]{catalan1}, and by noting that
  \begin{equation*}
    C_{2k+1,k+a}=\binom{2k}{k+a}-\binom{2k}{k+a-1}
    \quad\mbox{and}\quad
    C_{2k+1,k+a+1}=\binom{2k}{k+a+1}-\binom{2k}{k+a}\;.
  \end{equation*}
  Thus, the recurrence relation~(\ref{rec4B}) is a consequence
  of~(\ref{rec4C}), and we obtain that
  \begin{equation*}
    4^kB_{a,k}=\frac{1}{2}\left(C_{2k+1,k+a}-C_{2k+1,k+a+1}\right)\;.
  \end{equation*}
  Finally, we conclude that, for $k\in\N$,
  \begin{equation*}
    B_{0,k}=4^{-k}\left[\binom{2k}{k}-\frac{1}{2}
      \left[\binom{2k}{k-1}+\binom{2k}{k+1}\right]\right]
    =4^{-k}\left[\binom{2k}{k}-\binom{2k}{k+1}\right]
    =4^{-k}C_k\;,
  \end{equation*}
  which together with $B_{0,0}=1=C_0$ establishes~(\ref{B2catalan}).
\end{proof}
We need one more identity to determine the moments of $S_N$ in the
limit $N\to\infty$. This is where the discrepancy between $I_0(x)$ and
$\int_{-1}^xP_0(z)\dd z$ becomes useful.
\begin{lemma}\label{lem:idat0}
  For all $k\in\N_0$, we have
  \begin{equation*}
    \sum_{l=0}^k\left(\frac{1}{2l+1}+\frac{1}{2l+3}\right)b_{0,k}^l
    =\frac{2}{2k+1}-\frac{2}{2k+3}\;.
  \end{equation*}
\end{lemma}
\begin{proof}
  According to the partial fraction decomposition~(\ref{PFD}) in
  Proposition~\ref{propn:PFD}, we know
  \begin{equation*}
    m_{0,0}^k
    =-\sum_{l=0}^k\left(\frac{1}{2l+1}+\frac{1}{2l+3}\right)b_{0,k}^l\;.
  \end{equation*}
  On the other hand, since $I_0(x)=x-\im$ for $x\in[-1,1]$, we compute
  explicitly that
  \begin{equation*}
    m_{0,0}^k=\int_{-1}^1x^{2k}\left(I_0(x)\right)^2\dd x
    =\int_{-1}^1x^{2k}\left(x-\im\right)^2\dd x
    =\frac{2}{2k+3}-\frac{2}{2k+1}\;,
  \end{equation*}
  and the claimed result follows.
\end{proof}
We can finally describe the moments of $S_N$ in the limit $N\to\infty$.
\begin{propn}\label{propn:mom_conv}
  For all $k\in\N_0$, we have
  \begin{equation*}
    \int_{-1}^1x^{2k+1}S_N(x)\dd x=0
    \quad\mbox{for all }N\in\N\;,
  \end{equation*}
  and
  \begin{equation*}
    \lim_{N\to\infty}\int_{-1}^1x^{2k}S_N(x)\dd x
    =\frac{1}{2}\left(4^{-k}C_k\right)\;.
  \end{equation*}
\end{propn}
\begin{proof}
  Using Lemma~\ref{lem:oddmom} and the odd moments in~(\ref{moment0})
  and (\ref{moment1}), we obtain that
  \begin{equation*}
    \int_{-1}^1x^{2k+1}S_N(x)\dd x
    =N\left(\frac{2}{2k+3}-\frac{2}{2k+3}\right)=0
    \quad\mbox{for all }N\in\N\;,
  \end{equation*}
  as claimed. To determine the limit of the even
  moments, we fix $k\in\N_0$ and throughout, choose $N$
  sufficiently large. For $l\in\N_0$, we rewrite
  \begin{equation*}
    \sum_{n=1}^{N-1}\left(\frac{1}{2n-2l-1}-\frac{1}{2n+2l+3}\right)
    =\sum_{n=1}^{2l}\frac{1}{2n-2l-1}
    +\sum_{n=2l+1}^{N-1}\frac{1}{2n-2l-1}-\sum_{n=1}^{N-1}\frac{1}{2n+2l+3}\;,
  \end{equation*}
  and observe that
  \begin{equation*}
    \sum_{n=1}^{2l}\frac{1}{2n-2l-1}
    =\sum_{n=1}^{l}\frac{1}{2n-2l-1}+\sum_{n=1}^{l}\frac{1}{2(2l-n+1)-2l-1}
    =0
  \end{equation*}
  as well as
  \begin{align*}
    \sum_{n=2l+1}^{N-1}\frac{1}{2n-2l-1}-\sum_{n=1}^{N-1}\frac{1}{2n+2l+3}
    &=\sum_{n=1}^{N-2l-1}\frac{1}{2n+2l-1}-\sum_{n=1}^{N-1}\frac{1}{2n+2l+3}\\
    &=\frac{1}{2l+1}+\frac{1}{2l+3}-\sum_{n=N-2l-2}^{N-1}\frac{1}{2n+2l+3}
  \end{align*}
  to deduce that
  \begin{equation*}
    \sum_{n=1}^{N-1}\left(\frac{1}{2n-2l-1}-\frac{1}{2n+2l+3}\right)
    =\frac{1}{2l+1}+\frac{1}{2l+3}
    -\sum_{n=1}^{2l+2}\frac{1}{2N+2n-2l-3}\;.
  \end{equation*}  
  Applying Proposition~\ref{propn:PFD} and rearranging sums further
  yields
  \begin{align*}
    \sum_{n=1}^{N-1}m_{n,n}^k
    &=\sum_{n=1}^{N-1}\left(\sum_{l=0}^k\frac{b_{0,k}^l}{2n-2l-1}-
      \sum_{l=0}^k\frac{b_{0,k}^l}{2n+2l+3}\right)\\
    &=\sum_{l=0}^k \sum_{n=1}^{N-1}
      \left(\frac{1}{2n-2l-1}-\frac{1}{2n+2l+3}\right)b_{0,k}^l\\
    &=\sum_{l=0}^k\left(\frac{1}{2l+1}+\frac{1}{2l+3}\right)b_{0,k}^l
      -\sum_{l=0}^k\sum_{n=1}^{2l+2}\frac{b_{0,k}^l}{2N+2n-2l-3}\;.
  \end{align*}
  The even moments in~(\ref{moment0}) and (\ref{moment1}) as well as
  Lemma~\ref{lem:idat0} imply that
  \begin{equation*}
    \int_{-1}^1x^{2k}S_N(x)\dd x
    =N\left(\frac{1}{2k+1}-\frac{1}{2k+3}
      -\frac{1}{2}\sum_{n=1}^{N-1}m_{n,n}^k\right)
    =\frac{N}{2}\sum_{l=0}^k\sum_{n=1}^{2l+2}\frac{b_{0,k}^l}{2N+2n-2l-3}\;.
  \end{equation*}
  Finally, by~(\ref{B2catalan}) of Lemma~\ref{lem:catalan}, it follows
  that
  \begin{equation*}
    \lim_{N\to\infty}\int_{-1}^1x^{2k}S_N(x)\dd x
    =\frac{1}{2}\sum_{l=0}^k\sum_{n=1}^{2l+2}\frac{b_{0,k}^l}{2}
    =\frac{1}{2}\sum_{l=0}^k(l+1)b_{0,k}^l
    =\frac{1}{2}B_{0,k}=\frac{1}{2}\left(4^{-k}C_k\right)\;,
  \end{equation*}
  as required.
\end{proof}
The main result of this section is that the moments of $S_N$ converge
as $N\to\infty$ to the moments of a scaled semicircle.
\begin{propn}\label{propn:moments}
  Let $S\colon[-1,1]\to\R$ be given by
  \begin{equation*}
    S(x)=\frac{1}{\pi}\sqrt{1-x^2}
    \quad\mbox{for }x\in[-1,1]\;.
  \end{equation*}
  Then, for all $k\in\N_0$, we have
  \begin{equation*}
    \lim_{N\to\infty}\int_{-1}^1x^{k}S_N(x)\dd x
    =\int_{-1}^1x^{k}S(x)\dd x\;.
  \end{equation*}
\end{propn}
\begin{proof}
  It suffices to show that the moments of $S$
  are consistent with Proposition~\ref{propn:mom_conv}. Since $S$ is
  an even function on $[-1,1]$, we certainly have
  \begin{equation*}
    \int_{-1}^1x^{2k+1}S(x)\dd x=0
    \quad\mbox{for all }k\in\N_0\;.
  \end{equation*}
  Regarding the even moments, we follow~\cite[Section~2.1.1]{guionnet}
  and use the change of variable $x=\sin(\theta)$ where
  $\theta\in[-\pi/2,\pi/2]$ to write, for $k\in\N_0$,
  \begin{equation*}
    \int_{-1}^1x^{2k+2}S(x)\dd x
    =\frac{1}{\pi}\int_{-\pi/2}^{\pi/2}
    \sin^{2k+2}(\theta)\cos^2(\theta)\dd\theta\;.
  \end{equation*}
  By integration by parts, we have
  \begin{equation*}
    \int_{-\pi/2}^{\pi/2}\sin^{2k+3}(\theta)\sin(\theta)\dd\theta
    =\int_{-\pi/2}^{\pi/2}(2k+3)\sin^{2k+2}(\theta)\cos^2(\theta)\dd\theta\;,
  \end{equation*}
  and using $\cos^2(\theta)=1-\sin^2(\theta)$, we obtain
  \begin{equation*}
    \int_{-1}^1x^{2k+2}S(x)\dd x
    =\frac{1}{\pi}\int_{-\pi/2}^{\pi/2}\sin^{2k+2}(\theta)\dd\theta
    -(2k+3)\int_{-1}^1x^{2k+2}S(x)\dd x\;.
  \end{equation*}
  This together with applying integration by parts a second time
  implies that
  \begin{equation*}
    4^{k+1}\int_{-1}^1x^{2k+2}S(x)\dd x
    =\frac{4^{k+1}}{2k+4}\left(\frac{1}{\pi}
    \int_{-\pi/2}^{\pi/2}\sin^{2k+2}(\theta)\dd\theta\right)
    =\frac{2(2k+1)}{k+2}\left(4^k\int_{-1}^1x^{2k}S(x)\dd x\right)\;.
  \end{equation*}
  Since
  \begin{equation*}
    \int_{-1}^1S(x)\dd x
    =\frac{1}{\pi}\int_{-\pi/2}^{\pi/2}\cos^2(\theta)\dd\theta
    =\frac{1}{2}=\frac{1}{2} C_0
  \end{equation*}
  and as the Catalan numbers satisfy the recurrence relation
  \begin{equation*}
    C_{k+1}=\frac{2(2k+1)}{k+2}C_k\;,
  \end{equation*}
  it follows that
  \begin{equation*}
    4^k\int_{-1}^1x^{2k}S(x)\dd x=\frac{1}{2} C_k\;,
  \end{equation*}
  as needed.
\end{proof}
\section{Fluctuations for iterated Kolmogorov loops}
\label{fluctuations}
We establish a Christoffel--Darboux type formula for the integrals of
Legendre polynomials and put this together with the asymptotic
behaviours discussed in Section~\ref{legendre} as well as the moment
analysis performed in Section~\ref{moments} to prove
Theorem~\ref{thm:CLT4P}. Using the expression for the iterated
Kolmogorov loops given in Proposition~\ref{propn:kol_loop_leg}
and determined in Section~\ref{kol_loops},
we finally deduce Theorem~\ref{thm:CLT}.

The Christoffel--Darboux formula for Legendre polynomials,
see~\cite[Remark~5.2.2]{roy},
which is due to Christoffel~\cite{christoffel} and
Darboux~\cite{darboux}, states that, for $N\in\N$ and
$x,y\in[-1,1]$,
\begin{equation*}
  (x-y)\sum_{n=0}^N(2n+1)P_n(x)P_n(y)
  =(N+1)\left(P_{N+1}(x)P_N(y)-P_N(x)P_{N+1}(y)\right)\;.
\end{equation*}
The second identity in the lemma below can be considered as 
a Christoffel--Darboux type formula for the integrals of Legendre
polynomials.
\begin{propn}\label{propn:chris4I}
  Fix $x,y\in[-1,1]$ and set, for $n\in\Z$,
  \begin{equation}\label{defn:D}
    D_{n+1}(x,y)=I_{n+1}(x)I_n(y)-I_n(x)I_{n+1}(y)\;.
  \end{equation}
  Then we have
  \begin{equation}\label{CDstep1}
    (n+2)D_{n+1}(x,y)=(x-y)(2n+1)I_n(x)I_n(y)+(n-1)D_n(x,y)\;,
  \end{equation}
  and, for all $N\in\N$,
  \begin{equation}\label{CDformula}
    (x-y)\sum_{n=1}^N(2n+1)I_n(x)I_n(y)
    =N D_{N+1}(x,y)+2\sum_{n=1}^ND_{n+1}(x,y)\;.
  \end{equation}
\end{propn}
\begin{proof}
  Following~\cite[Proof of Theorem~3.2.2]{gabor}, we use to recursion
  formula in Lemma~\ref{lem:rec4I} to deduce
  \begin{align*}
    &(n+2)\left(I_{n+1}(x)I_n(y)-I_n(x)I_{n+1}(y)\right)\\
    &\qquad=\left((2n+1)xI_n(x)-(n-1)I_{n-1}(x)\right)I_n(y)
      -I_n(x)\left((2n+1)yI_n(y)-(n-1)I_{n-1}(y)\right)\\
    &\qquad=(x-y)(2n+1)I_n(x)I_n(y)
      +(n-1)\left(I_n(x)I_{n-1}(y)-I_{n-1}(x)I_n(y)\right)\;,
  \end{align*}
  which establishes~(\ref{CDstep1}). Applying this identity, we
  further obtain
  \begin{align*}
    (x-y)\sum_{n=1}^N(2n+1)I_n(x)I_n(y)
    &=\sum_{n=1}^N(n+2)D_{n+1}(x,y)-\sum_{n=1}^N(n-1)D_n(x,y)\\
    &=(N+2)D_{N+1}(x,y)+\sum_{n=2}^N(n+1)D_n(x,y)-\sum_{n=2}^N(n-1)D_n(x,y)\\
    &=ND_{N+1}(x,y)+2\sum_{n=1}^ND_{n+1}(x,y)\;,
  \end{align*}
  as claimed.
\end{proof}
This Christoffel--Darboux type formula enters our analysis in the
proof of the following lemma.
\begin{lemma}\label{lem:offdiagconv}
  Fix $x,y\in[-1,1]$. Then, for all $\alpha\in\R$ with $\alpha<1$,
  we have
  \begin{equation*}
    \lim_{N\to\infty}(x-y)N^{\alpha+1}
    \sum_{n=N}^\infty(2n+1)I_n(x)I_n(y)=0\;.
  \end{equation*}
\end{lemma}
\begin{proof}
  The result is trivially true if $x\in\{-1,1\}$ or $y\in\{-1,1\}$
  because $I_n(-1)=I_n(1)=0$ for all $n\in\N$, cf.~(\ref{Iboundary}).
  Let us now suppose that $x,y\in(-1,1)$ and choose
  $N,M\in\N$ with $N<M$. From Proposition~\ref{propn:chris4I}, it
  follows that
  \begin{equation}\label{cut_CD}
    (x-y)\sum_{n=N}^M(2n+1)I_n(x)I_n(y)
    =MD_{M+1}(x,y)+2\sum_{n=N}^MD_{n+1}(x,y)-(N-1)D_N(x,y)\;.
  \end{equation}
  The asymptotic behaviour~(\ref{intasymp}) given by the Darboux formula
  implies that there exists a positive constant $K\in\R$,
  depending on $x$ and $y$, such that, for all $n$ sufficiently large,
  \begin{equation*}
    \left|P_n^{(-1,-1)}(x)\right|\leq \frac{K}{2}n^{-\frac{1}{2}}
    \quad\mbox{and}\quad
    \left|P_n^{(-1,-1)}(y)\right|\leq \frac{K}{2}n^{-\frac{1}{2}}\;.
  \end{equation*}
  Since the Jacobi polynomial $P_{n+1}^{(-1,-1)}$ and
  the integral $I_n$ are
  related by $P_{n+1}^{(-1,-1)}=\frac{1}{2}n I_n$ for $n\in\N$,
  see~(\ref{I2jacobi}), we obtain that, for $n$ large enough,
  \begin{equation*}
    \left|I_n(x)\right|\leq K n^{-\frac{3}{2}}
    \quad\mbox{and}\quad
    \left|I_n(y)\right|\leq K n^{-\frac{3}{2}}\;.
  \end{equation*}
  From the definition~(\ref{defn:D}) of $D_{n+1}$ we deduce that,
  for $n$ sufficiently large,
  \begin{equation}\label{Destimate}
    \left|D_{n+1}(x,y)\right|\leq 2K^2n^{-3}\;.
  \end{equation}
  In particular, this shows
  \begin{equation*}
    \lim_{M\to\infty}MD_{M+1}(x,y)=0\;,
  \end{equation*}
  and, by the integral test, that, for $N$ large enough,
  \begin{equation*}
    \left|\sum_{n=N}^\infty D_{n+1}(x,y)\right|
    \leq 2K^2\sum_{n=N}^\infty\frac{1}{n^3}
    \leq 2K^2\left(\frac{1}{N^3}+\int_N^\infty z^{-3}\dd z\right)
    =\frac{2K^2}{N^3}+\frac{K^2}{N^2}\;.
  \end{equation*}
  By~(\ref{cut_CD}), these estimates establish
  \begin{equation*}
    (x-y)\sum_{n=N}^\infty(2n+1)I_n(x)I_n(y)
    =2\sum_{n=N}^\infty D_{n+1}(x,y)-(N-1)D_N(x,y)
  \end{equation*}
  as well as
  \begin{equation*}
    \left|(x-y)N^{\alpha+1}\sum_{n=N}^\infty(2n+1)I_n(x)I_n(y)\right|
    \leq 4K^2N^{\alpha-2}+2K^2N^{\alpha-1}
    +N^{\alpha+1}(N-1)\left|D_N(x,y)\right|\;.
  \end{equation*}
  Provided that $\alpha<1$, we have $N^{\alpha-1}\to 0$ and
  $N^{\alpha-2}\to 0$ as $N\to\infty$, and since~(\ref{Destimate})
  further yields
  \begin{equation*}
    \lim_{N\to\infty}N^{\alpha+1}(N-1) D_N(x,y)=0
    \quad\mbox{for }\alpha<1\;,
  \end{equation*}
  the claimed result follows.
\end{proof}
The reason why the Christoffel--Darboux type formula~(\ref{CDformula})
allows us to prove Lemma~\ref{lem:offdiagconv} is that as argued in
the above proof, the asymptotic~(\ref{intasymp}) implies that
$D_{n+1}(x,y)$ is of order $O(n^{-3})$ as $n\to\infty$, whereas
$(2n+1)I_n(x)I_n(y)$ is only seen to be of order $O(n^{-2})$
as $n\to\infty$.

We use Lemma~\ref{lem:offdiagconv} in the proof of
Theorem~\ref{thm:CLT4P} to show the convergence away from the
diagonal, while the following lemma provides what is needed to establish
locally uniform convergence on the diagonal. The convergence
of moments, cf. Proposition~\ref{propn:moments}, then characterises the
limit uniquely.
\begin{lemma}\label{lem:locunif_bound}
  Fix $\eps>0$. The families
  \begin{equation*}
    \left\{
      N\sum_{n=N}^\infty(2n+1)I_n(x)I_n(y)\colon
      N\in\N\mbox{ and }x,y\in[-1+\eps,1-\eps]
    \right\}
  \end{equation*}
  and
  \begin{equation*}
    \left\{
      (N+1) P_N(x)P_{N+1}(x)\colon N\in\N\mbox{ and }x\in[-1+\eps,1-\eps]
    \right\}
  \end{equation*}
  are uniformly bounded.
\end{lemma}
\begin{proof}
  As a consequence of the estimate~(\ref{intasymp}) from the
  Darboux formula, there exists a positive constant $K\in\R$ such
  that, for $n$ sufficiently large, we have
  \begin{equation*}
    \left|P_n^{(-1,-1)}(x)\right|\leq \frac{K}{2}n^{-\frac{1}{2}}
    \quad\mbox{uniformly in }x\in[-1+\eps,1-\eps]\;.
  \end{equation*}
  Due to the relation $P_{n+1}^{(-1,-1)}=\frac{1}{2}n I_n$ for
  $n\in\N$, this implies that, for $n$ large enough,
  \begin{equation*}
    \left|I_n(x)\right|\leq K n^{-\frac{3}{2}}
    \quad\mbox{uniformly in }x\in[-1+\eps,1-\eps]\;.
  \end{equation*}
  We deduce that, uniformly in $x,y\in[-1+\eps,1-\eps]$ and for $N$
  sufficiently large,
  \begin{equation*}
    \left|N\sum_{n=N}^\infty(2n+1)I_n(x)I_n(y)\right|
    \leq 3NK^2\sum_{n=N}^\infty\frac{1}{n^2}
    \leq 3NK^2\left(\frac{1}{N^2}+\int_N^\infty z^{-2}\dd z\right)
    \leq 6K^2\;,
  \end{equation*}
  which establishes the uniform boundedness of the first family. We
  argue in a similar way for the second family. By the
  asymptotic~(\ref{legasymp}) obtained from the Darboux formula and
  since $P_n^{(0,0)}=P_n$ for $n\in\N_0$, see~(\ref{L2jacobi}), there
  exists a positive constant $L\in\R$ such that,
  for $N$ sufficiently large,
  \begin{equation*}
    \left|P_N(x)\right|\leq L N^{-\frac{1}{2}}
    \quad\mbox{uniformly in }x\in[-1+\eps,1-\eps]\;.
  \end{equation*}
  Thus, for $N$ large enough,
  \begin{equation*}
    \left|(N+1) P_N(x)P_{N+1}(x)\right| \leq 2L^2
    \quad\mbox{uniformly in }x\in[-1+\eps,1-\eps]\;,
  \end{equation*}
  and the uniform boundedness of the second family follows.
\end{proof}
We finally combine our results to give the proof of
Theorem~\ref{thm:CLT4P}.
\begin{proof}[Proof of Theorem~\ref{thm:CLT4P}]
  As argued for the shifted Legendre polynomials
  in the proof of Theorem~\ref{thm:LLN} in
  Section~\ref{kol_loops}, the
  polarised Parseval identity shows that, for $x,y\in[-1,1]$,
  \begin{equation*}
    \min(1+x,1+y) =\sum_{n=0}^\infty
    \frac{2n+1}{2}\int_{-1}^xP_n(z)\dd z\int_{-1}^yP_n(z)\dd z\;.
  \end{equation*}
  Therefore, $R_N(x,y)$ defined by~(\ref{defn:R}) can be expressed as,
  for $N\in\N$,
  \begin{equation}\label{Rinfinitesum}
    R_N(x,y) =N \sum_{n=N}^\infty
    \frac{2n+1}{2}\int_{-1}^xP_n(z)\dd z\int_{-1}^yP_n(z)\dd z
    =\frac{1}{2}N\sum_{n=N}^\infty(2n+1)I_n(x)I_n(y)\;.
  \end{equation}
  Hence, if $x,y\in[-1,1]$ with $x\not=y$ then
  Lemma~\ref{lem:offdiagconv} applied for
  $\alpha=0$ implies
  that $R_N(x,y)\to 0$ as $N\to\infty$, which establishes the desired
  convergence away from the diagonal. It remains to consider the
  diagonal case $x=y$. As in Section~\ref{moments}, we consider the
  functions $S_N\colon[-1,1]\to\R$ defined by
  \begin{equation*}
    S_N(x)=R_N(x,x)
    \quad\mbox{for }x\in[-1,1]\;.
  \end{equation*}
  Using the expression for $S_N$, which follows from~(\ref{defn:R})
  for $R_N$, that
  \begin{equation*}
    S_N(x)=N\left(1+x-\frac{1}{2}(1+x)^2
      -\sum_{n=1}^{N-1}\frac{2n+1}{2}
      \left(\int_{-1}^xP_n(z)\dd z\right)^2\right)\;,
  \end{equation*}
  and the relation~(\ref{LegInt}), we compute
  \begin{align*}
    \frac{\db}{\db x}S_N(x)
    &=N\left(1-(1+x)-\sum_{n=1}^{N-1}(2n+1)P_n(x)\int_{-1}^xP_n(z)\dd z\right)\\
    &=-N\left(x
      +\sum_{n=1}^{N-1}P_n(x)\left(P_{n+1}(x)-P_{n-1}(x)\right)\right)\\
    &=-N P_{N-1}(x)P_N(x)\;.
  \end{align*}
  By Lemma~\ref{lem:locunif_bound}, it follows that
  the sequence $(S_N)_{N\in\N}$ is uniformly bounded and uniformly
  Lipschitz on $[-1+\eps,1-\eps]$ for $\eps>0$. The Arzel{\`a}--Ascoli
  theorem implies that $(S_N)_{N\in\N}$ is locally uniformly
  convergent on $(-1,1)$ and we deduce that $(S_N)_{N\in\N}$ converges to a
  continuous function on $(-1,1)$.   
  Thus, the limit function is uniquely identified by
  Proposition~\ref{propn:moments} and since
  $S_N(-1)=S_N(1)=0$ for all $N\in\N$, we conclude that, for all
  $x\in[-1,1]$,
  \begin{equation*}
    \lim_{N\to\infty}R_N(x,x)=\lim_{N\to\infty}S_N(x)
    =S(x)=\frac{1}{\pi}\sqrt{1-x^2}\;,
  \end{equation*}
  as required.
\end{proof}
We obtain Theorem~\ref{thm:CLT}
as a consequence of Theorem~\ref{thm:CLT4P}.
\begin{proof}[Proof of Theorem~\ref{thm:CLT}]
  As established in the proof of Lemma~\ref{lem:covariance},
  the covariance function of the process $(L_t^N)_{t\in[0,1]}$ is
  $C_N$ and hence, the fluctuation process $(F_t^N)_{t\in[0,1]}$
  defined by $F_t^N=\sqrt{N}L_t^N$ has covariance $NC_N$.
  Moreover, for $s,t\in[0,1]$, we have
  \begin{align}\label{R2C}
    \begin{aligned}
      R_N(2s-1,2t-1)
      &=2N\left(\min(s,t)-\sum_{n=0}^{N-1}\left(2n+1\right)
        \int_{0}^sQ_n(r)\dd r
        \int_{0}^tQ_n(r)\dd r\right)\\
      &=2NC_N(s,t)\;.
    \end{aligned}
  \end{align}
  By Theorem~\ref{thm:CLT4P}, it follows that, for $s,t\in[0,1]$ fixed,
  \begin{equation*}
    \lim_{N\to\infty}NC_N(s,t)
    =\frac{1}{2}\lim_{N\to\infty}R_N(2s-1,2t-1)=
    \begin{cases}
      \frac{1}{\pi}\sqrt{t(1-t)} & \mbox{if } s=t\\
      0 & \mbox{if } s\not=t
    \end{cases}\;.
  \end{equation*}
  Thus, for any $k\in\N$ and any $t_1,\dots,t_k\in[0,1]$, the
  characteristic functions of the Gaussian random vectors
  $(F_{t_1}^N,\dots,F_{t_k}^N)$ converge pointwise as $N\to\infty$
  to the characteristic function of 
  the Gaussian random vector $(F_{t_1},\dots,F_{t_k})$. By
  L{\'e}vy's continuity theorem, this implies the claimed convergence in
  finite dimensional distributions.
\end{proof}
We close with the observation that a slightly modified
analysis even allows us to deduce a non-trivial
bound on the scale of the decorrelation.
\begin{remark}\label{rem:decorr}
  Fix $x\in(-1,1)$. For $\beta\in\R$ with $\beta>0$
  and $y\in\R\setminus\{0\}$, set
  \begin{equation*}
    y_N=x+N^{-\beta}y\;.
  \end{equation*}
  We note that the sequence
  $(y_N)_{N\in\N}$ converges monotonically to $x$ as $N\to\infty$.
  Since the asymptotic estimate~(\ref{intasymp})
  is uniform in
  $\theta\in[\eps,\pi-\eps]$ for $\eps>0$, the argument presented
  to prove Lemma~\ref{lem:offdiagconv} can be improved to show
  the existence of a positive constant $K\in\R$ such that,
  for $n$ and $N$ large enough,
  \begin{equation*}
    \left|D_{n+1}(x,y_N)\right|\leq 2K^2n^{-3}\;.
  \end{equation*}
  As in the proof of Lemma~\ref{lem:offdiagconv}, this implies
  \begin{equation*}
    \left|(x-y_N)N\sum_{n=N}^\infty(2n+1)I_n(x)I_n(y_N)\right|
    \leq \frac{4K^2}{N^2}+\frac{4K^2}{N}
    \leq \frac{8K^2}{N}\;,
  \end{equation*}
  and therefore, according to~(\ref{Rinfinitesum}), that
  \begin{equation*}
    \left|R_N(x,y_N)\right|
    \leq\frac{4K^2}{\left|x-y_N\right|N}
    =\frac{4K^2}{|y|}N^{\beta-1}\;.
  \end{equation*}
  Hence, as long as $\beta<1$, we are guaranteed that
  \begin{equation*}
    \lim_{N\to\infty}R_N(x,y_N)=0\;.
  \end{equation*}
  Due to~(\ref{R2C}), this rewrites
  in terms of the covariance function $C_N$,
  for $s\in(0,1)$ and $t\in\R\setminus\{0\}$, as
  \begin{equation*}
    \lim_{N\to\infty}N C_N\left(s,s+N^{-\beta}t\right)
    =\frac{1}{2}\lim_{N\to\infty}R_N(2s-1,2s-1+N^{-\beta}2t)
    =0
    \quad\mbox{for }\beta<1\;,
  \end{equation*}
  which provides a bound on the decorrelation scale for the fluctuation
  processes $(F_t^N)_{t\in[0,1]}$.
\end{remark}
\appendix
\section{Recurrence relations for the
  partial fraction coefficients}
\label{app:recurrence}
We continue the analysis started in the proof of
Proposition~\ref{propn:PFD} to determine recurrence relations for the
coefficients $b_{a,k}^l$ and to include them for completeness. Due
to the symmetry property~(\ref{PFD:sym}) of
Proposition~\ref{propn:PFD}, we restrict our attention to the family
$\{b_{a,k}^l\in\R\colon a,k,l\in\N_0\}$. As discussed when settling the
base case for the inductive proof of Proposition~\ref{propn:PFD}, we
have, cf.~(\ref{coeff4k0}),
\begin{equation}\label{rec_initial}
  b_{0,0}^0=1\;,
  \enspace
  b_{1,0}^0=-\frac{1}{2}
  \quad\mbox{and}\quad
  b_{a,0}^l=0
  \mbox{ otherwise}\;.
\end{equation}
These are the initial conditions for our recursion. The recurrence
relations for the coefficients $b_{a,k}^l$
are deduced, by use of the Heaviside cover-up method, from the
expression, for $a\in\N_0$,
\begin{align}\label{app:mom_rec}
  \begin{aligned}
    m_{n-a,n+a}^{k}
    &=\frac{(2n+1)(n-a+2)(n+a+2)}{(2n-2a+1)(2n+2a+1)(2n+3)}
    \left(\sum_{l=0}^{k-1}\frac{b_{a,k-1}^l}{2n-2l+1}-
      \sum_{l=0}^{k-1}\frac{b_{a,k-1}^l}{2n+2l+5}\right)\\
    &\quad\enspace+
    \frac{(2n+1)(n-a-1)(n+a-1)}{(2n-2a+1)(2n+2a+1)(2n-1)}
    \left(\sum_{l=0}^{k-1}\frac{b_{a,k-1}^l}{2n-2l-3}-
      \sum_{l=0}^{k-1}\frac{b_{a,k-1}^l}{2n+2l+1}\right)\\
    &\quad\enspace+
    \frac{(n-a+2)(n+a-1)}{(2n-2a+1)(2n+2a+1)}
    \left(\sum_{l=0}^{k-1}\frac{b_{a-1,k-1}^l}{2n-2l-1}-
      \sum_{l=0}^{k-1}\frac{b_{a-1,k-1}^l}{2n+2l+3}\right)\\
    &\quad\enspace+
    \frac{(n-a-1)(n+a+2)}{(2n-2a+1)(2n+2a+1)}
    \left(\sum_{l=0}^{k-1}\frac{b_{a+1,k-1}^l}{2n-2l-1}-
      \sum_{l=0}^{k-1}\frac{b_{a+1,k-1}^l}{2n+2l+3}\right)\;,
  \end{aligned}
\end{align}
which is a consequence of Lemma~\ref{lem:mom_rec} applied with $p=n-a$
and $q=n+a$ and Proposition~\ref{propn:PFD}. When
employing the Heaviside cover-up method, we need to be careful
about factors which could occur quadratically in the denominators.
By the partial fraction decomposition~(\ref{PFD}) for
$m_{n-a,n+a}^{k}$ we are justified to write
\begin{equation}\label{babuse}
  b_{a,k}^l=\left.(2n-2l-1)m_{n-a,n+a}^k\right|_{n=l+1/2}\;.
\end{equation}
If $l\not=0$ and $l\not=a-1$, it follows from~(\ref{app:mom_rec})
and~(\ref{babuse}) that, for $k\in\N$,
\begin{align*}
  b_{a,k}^l
  &=\frac{(l+1)(2l-2a+5)(2l+2a+5)}{16(l+2)(l-a+1)(l+a+1)}b_{a,k-1}^{l+1}+
    \frac{(l+1)(2l-2a-1)(2l+2a-1)}{16l(l-a+1)(l+a+1)}b_{a,k-1}^{l-1}\\
  &\qquad+
    \frac{(2l-2a+5)(2l+2a-1)}{16(l-a+1)(l+a+1)}b_{a-1,k-1}^l+
    \frac{(2l-2a-1)(2l+2a+5)}{16(l-a+1)(l+a+1)}b_{a+1,k-1}^l\;.
\end{align*}
For $l=0$, we need to treat the two cases $a=1$ and $a\not=1$
separately. If $a\not=1$, we obtain in the same way as before that
\begin{align}\label{recbak0}
  \begin{aligned}
    b_{a,k}^0
    &=\frac{(2a-5)(2a+5)}{32(a-1)(a+1)}b_{a,k-1}^1-
    \frac{(2a+1)(2a-1)}{8(a-1)(a+1)}\sum_{l=0}^{k-1}\frac{b_{a,k-1}^l}{l+1}\\
    &\qquad+
    \frac{(2a-5)(2a-1)}{16(a-1)(a+1)}b_{a-1,k-1}^0+
    \frac{(2a+1)(2a+5)}{16(a-1)(a+1)}b_{a+1,k-1}^0\;,
  \end{aligned}
\end{align}
which for $a=0$ reduces to
\begin{equation*}
  b_{0,k}^0=\frac{25}{32}b_{0,k-1}^1
  -\frac{1}{8}\sum_{l=0}^{k-1}\frac{b_{0,k-1}^l}{l+1}
  -\frac{5}{8}b_{1,k-1}^0\;.
\end{equation*}
If $a=1$, we use the property~(\ref{no_blow_up})
established as part of the proof of Proposition~\ref{propn:PFD}
to show that the blow-up term $(2n-1)^{-1}$ appearing in the
Heaviside cover-up method vanishes, and we deduce
\begin{align*}
  b_{1,k}^0
  &=\frac{21}{32}\left(\frac{b_{1,k-1}^0}{3}-\frac{b_{1,k-1}^1}{8}
    -\sum_{l=2}^{k-1}\frac{(l+1)b_{1,k-1}^l}{(l-1)(l+3)}\right)\\
  &\qquad-
    \frac{3}{16}\left(\frac{b_{0,k-1}^0}{4}
    +\sum_{l=1}^{k-1}\frac{(l+1)b_{0,k-1}^l}{l(l+2)}\right)+
    \frac{21}{16}
    \left(\frac{b_{2,k-1}^0}{4}
    +\sum_{l=1}^{k-1}\frac{(l+1)b_{2,k-1}^l}{l(l+2)}\right)\;.
\end{align*}
It remains to consider the case $l=a-1$ for $a\geq 2$. As above, we
use the property~(\ref{no_blow_up}) to show that the potential blow-up
term $(2n-2a+1)^{-1}$ vanishes, and we derive
\begin{align*}
  b_{a,k}^{a-1}
  &=-\frac{3(4a+3)}{16(a+1)}
    \left(\frac{b_{a,k-1}^a}{4(a+1)}+
    \sum_{l=0,l\not=a}^{k-1}\frac{(l+1)b_{a,k-1}^l}{(l-a)(l+a+2)}\right)\\
  &\qquad+
    \frac{3(4a-3)}{16(a-1)}
    \left(\frac{b_{a,k-1}^{a-2}}{4(a-1)}+
    \sum_{l=0,l\not=a-2}^{k-1}\frac{(l+1)b_{a,k-1}^l}{(l+a)(l-a+2)}\right)\\
  &\qquad-
    \frac{3(4a-3)}{16a}
    \left(\frac{b_{a-1,k-1}^{a-1}}{4a}+
    \sum_{l=0,l\not=a-1}^{k-1}\frac{(l+1)b_{a-1,k-1}^l}{(l-a+1)(l+a+1)}\right)\\
  &\qquad+
    \frac{3(4a+3)}{16a}
    \left(\frac{b_{a+1,k-1}^{a-1}}{4a}+
    \sum_{l=0,l\not=a-1}^{k-1}\frac{(l+1)b_{a+1,k-1}^l}{(l-a+1)(l+a+1)}\right)\;.
\end{align*}
While especially the recurrence relations for the cases $l=0$ and
$l=a-1$ are not particularly nice, we have enough relations
to uniquely
determine the family $\{b_{a,k}^l\in\R\colon a,k,l\in\N_0\}$ of
coefficients from~(\ref{rec_initial}) by recursion over $k\in\N_0$. It
is even possible to use these recurrence relations for the coefficients
$b_{a,k}^l$ and~(\ref{no_blow_up})
to prove the recurrence relation for the sums $B_{a,k}$
given in Proposition~\ref{propn:rec_rel4B} by brute force.
However, this approach needs a lot of care and is less elegant.
Though, it could be of interest to investigate if the above recurrence
relations could be significantly simplified. For instance, we note
that according to~(\ref{PFD:bound}) of Proposition~\ref{propn:PFD},
for $a\in\N$,
\begin{equation*}
  \sum_{l=0}^{k-1}\frac{b_{a,k-1}^l}{l+1}=-\frac{b_{0,k-1}^{a-1}}{2a}\;,
\end{equation*}
which implies that~(\ref{recbak0}) for $a\geq 2$ is equivalent to
\begin{align*}\label{recbak0}
  b_{a,k}^0
  &=\frac{(2a-5)(2a+5)}{32(a-1)(a+1)}b_{a,k-1}^1+
    \frac{(2a+1)(2a-1)}{16a(a-1)(a+1)}b_{0,k-1}^{a-1}\\
  &\qquad+
    \frac{(2a-5)(2a-1)}{16(a-1)(a+1)}b_{a-1,k-1}^0+
    \frac{(2a+1)(2a+5)}{16(a-1)(a+1)}b_{a+1,k-1}^0\;.
\end{align*}
We close by remarking that the coefficients $b_{a,k}^l$ can be easily
generated using {\sc Mathematica} by assigning the
appropriate values to {\tt a} and {\tt k}, calling the command
\begin{verbatim}
  Apart[FindSequenceFunction[Table[
    (2n+1)*Integrate[x^(2k)*Integrate[LegendreP[n-a,z],{z,-1,x}]*
      Integrate[LegendreP[n+a,z],{z,-1,x}],{x,-1,1}],{n,a+1,a+20}],n-a]]
\end{verbatim}
and reading off the coefficients
$b_{a,k}^l$ for $l\in\{0,\dots,k\}$,
where the upper bound of {\tt n} needs
to be increased for large values of {\tt k}.
\bibliographystyle{plain}
\bibliography{references}

\end{document}